\documentclass[12pt,letterpaper]{article}

\usepackage{times}
\usepackage[T1]{fontenc}
\usepackage{mathrsfs}
\usepackage{latexsym}
\usepackage[dvips]{graphics}
\usepackage{epsfig}
\usepackage{amsmath,amsfonts,amsthm,amssymb,amscd}
\input amssym.def
\input amssym.tex

\newtheorem{thm}{Theorem}[section]

\newtheorem{cor}[thm]{Corollary}
\newtheorem{lem}[thm]{Lemma}
\newtheorem{prop}[thm]{Proposition}

\newtheorem{defi}[thm]{Definition}

\newcommand\be{\begin{equation}}
\newcommand\ee{\end{equation}}
\newcommand\bea{\begin{eqnarray}}
\newcommand\eea{\end{eqnarray}}
\newcommand\bi{\begin{itemize}}
\newcommand\ei{\end{itemize}}
\newcommand\ben{\begin{enumerate}}
\newcommand\een{\end{enumerate}}
\newcommand\bc{\begin{center}}
\newcommand\ec{\end{center}}
\newcommand\ba{\begin{array}}
\newcommand\ea{\end{array}}

\begin{document}

\title{Diffractive Theorems for the Wave Equation with Inverse Square Potential}
\author{Randy Z. Qian}
\date{\today}

\maketitle

\begin{abstract}
We first establish the presence of a diffractive front in the fundamental solution of the wave operator with a diract delta intial condition in two dimensional euclidean space caused by the potentials perturbation on the spherical laplacian. This motivates a result which restricts the propagation of singularities for the wave operator with a more general potential to precisely these diffractive fronts higher dimensional euclidean spaces. This is proven using microlocal energy estimates.
\end{abstract}

\section{Introduction}

This paper deals with the diffractive phenomenon exhibited by the wave equation when singularities in its distributional solution encounter an inverse square potential. Potentials of this order are of broad interest since they form the border line case for the existence of global in time estimates for the wave and Schr\"odinger equations (see \cite{gorgiev_vis} \cite{rodnianski_schlag}); recent work has been done by Burq, Planchon, Stalker, and Tahvildar-Zadeh to establish Strichartz estimates for these equations with such a potential in \cite{Burq_et}. Inverse square potentials also appear in physical equations, including a recast version of the Dirac equation with Coulomb potential (see \cite{case}) and the linearized perturbations of the Schwarzschild solution among others for the Einstein equations (see \cite{regge_wheeler} \cite{zerilli}). Furthermore, the interaction of singularities in solutions of PDE with singular spaces have long been known to cause diffraction; we will survey the work done in this direction before stating some diffractive results for inverse square potentials. The mathematical interest of this potential in all of these cases largely comes from the $-2$ radial order of homogeneity common to both the inverse square potential and the Laplacian. 

In Euclidean spacetime, $\mathbb{R}^{n-1} \times \mathbb{R}=\{(x,t)\}$, the simplest version of our scenario is the wave equation with an inverse square potential $\frac{a}{r^2}$ (where $r$ is radial distance in $x$ from $x=0$) which is described by:
\be\label{euclidean_diffraction}
\left\{ \begin{array}{rcl}
\square u(x,t) + \frac{a}{r^2} u(x,t) &=& 0  \qquad\qquad\qquad (a>0,\; \text{a constant})\\
u_t(x,0)&=&\delta(x-x_0)\\
u(x,0)&=&0
\end{array}\right.
\ee

The ``box'' symbol: $\square$ is the d'Alembert operator defined as:

\be
\square := \frac {\partial^2}{\partial t^2}-\sum_{i=1}^{n-1}{\frac{\partial^2}{\partial x_i^2}}
\ee

where the second term is the familiar Laplacian $\Delta:=-\sum_{i=1}^{n-1}{\frac{\partial^2}{\partial x_i^2}}$ in the $n-1$ spatial variables. $\delta(x-x_0)$ is a Dirac delta function giving an initial condition with its singular peak centered at $x_0$. Our interest is in finding the set of singularities of the distributional (weak) solution $u$ in space-time. This set is denoted singsupp$(u)$ and defined as the places where $u$ is not $C^\infty$.

To illustrate, consider figure 1 below which qualitatively shows the behavior of the solution $u$ in the spacetime $\mathbb{R}^2\times \mathbb{R}$. In the forward time direction, the behavior of the initial data $\delta(x-x_0)$ is like an ``explosion'' at $t=0$ which sends out a primary spherical front of singularities at unit speed that sweeps out a cone (with profile slope 1) in space-time, accounting for the outer cone which is between regions I and II in the figure; this behaviour is the same as what would occur for the free wave equation, $\square u = 0$, under the same initial conditions. The inner cone between regions II and III is a diffractive front in the solution brought about by the initial front striking the potential at the origin; this is a qualitatively different phenomenon from the singularities in the free case. 
\begin{figure}[htp]
\begin{center}
\fbox{\parbox[c][2.1in][l]{2.4in} %top, 2in deep, v-centered, 4in wide
{\leftline{\includegraphics[scale=0.5]{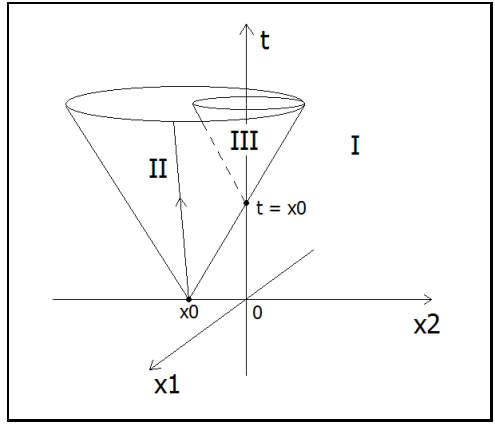}}}
}
\end{center}
\caption{main front (with embedded bicharacteristic) and diffractive front.}
\end{figure}

A similar diffractive front is the subject of several other results. The first rigorous treatment of diffraction for the wave equation was made by Sommerfeld in \cite{sommerfeld} for the case of a singular front encountering a slit in a plane, causing diffraction into the shadow region of the slit. Friedlander extended this to cone obstacles in \cite{friedlander} followed by Cheeger and Taylor who explicitly dealt with the case of general product cones in \cite{cheeger_taylor}. The potential $\frac{a}{r^2}$ can be also thought of as a singular feature of our space at the origin, and by employing the separation of variable methods used by the authors above, we can deduce the following theorem about the explicit solution:

\begin{thm} {\emph{Existence of diffraction:}}\label{theorem1}
There is an explicit diffractive front corresponding to the inner cone in figure 1 in $u$, the explicit solution of (\ref{euclidean_diffraction}) in $\mathbb{R}^2$.
\end{thm}

In more general geometric contexts, we can't use separation of variables to achieve a similar explicit solution; energy methods are preferred. Work by Lebeau gives a diffractive result in the analytic setting for a broad class of manifolds in \cite{lebeau}. Melrose and Wunsch have established results for conic manifolds in \cite{melrose_wunsch} in the smooth setting, and along with Vasy who studied manifolds with corners in the smooth setting \cite{vasy}, they have recently published results in the case of edge manifolds \cite{melrose_wunsch_vasy}. Using (microlocal) energy techniques employed by these authors, we are led to a result about the propagation of smoothness around points of the form $(r=0, t)$ where the potential is concentrated:

\begin{figure}[htp]
\begin{center}
\fbox{\parbox[c][2.2in][l]{1.45in} %top, 2in deep, v-centered, 4in wide
{\leftline{\includegraphics[scale=0.7]{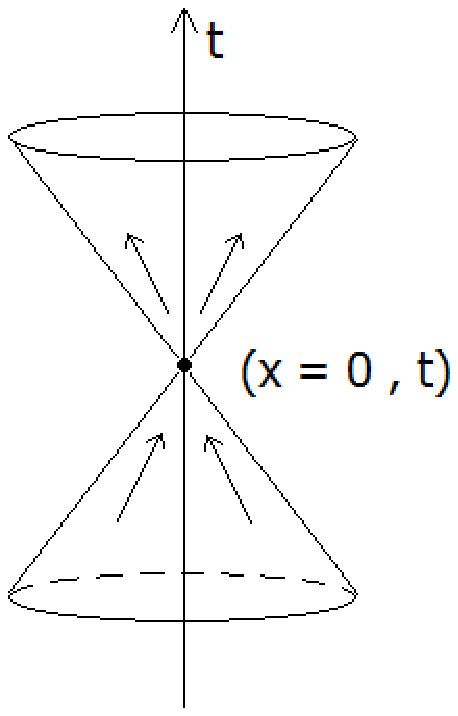}}}
}
\end{center}
\caption{incoming smoothness implies outgoing smoothness.}
\label{smoothin_smoothout}
\end{figure}

\begin{thm} {\emph{Propagation of smoothness:}} \label{theorem2}
Let $u$ be a solution of

\be
\left\{ \begin{array}{rcl}
\square u(x,t) + \frac{f(r,\theta)}{r^2} u(x,t) &=& 0  \qquad\qquad\, f \in C^\infty, \|f\|_\infty < \infty, f\ge -\left(\frac{n-2}{2}\right)\\
\end{array}\right.
\ee

If there are no singularities of $u$ going into the point $(0,t)$ originating from this point's backward cone of influence, then there are no singularities coming out of its forward diffractive cone as seen in figure \ref{smoothin_smoothout}.
\end{thm}

Note that there are no other possible incoming or outgoing singularities at this point except for those in the two cones depicted since they must all travel with unit speed due to a standard result from microlocal analysis. Also note that the in Theorem \ref{theorem1}, the single singularity traveling along the bicharacteristic toward the origin is enough to cause the shower of singularities that make up the diffractive front. Since energy estimates tell us about the propagation of regularity, we can't hope for a better result in this context than the propagation of smoothness along the light cone out of the singularity given smoothness along the entire incoming cone. Results of this type can be thought of as a microlocal version of the local smoothing estimates for the wave equation seen in \cite{Burq_et} that take a further phase variable into account.

Because of the arbitrary bounded smooth function in the potential term, we can't separate variables to find an explicit solution which directly exhibits diffractive behavior. Energy methods become our primary tool, and the resulting theorems deal with the propagation of smoothness, which in turn restricts the behavior of singularities. The result above restricts the outcome of interactions between singularities from the solution and the potential's singular point to only the emission of a diffractive front immediately following an incoming singularity; we will not see the spontaneous emission of singularities from the potential, nor will singularties be trapped and released at a later time. 

\section{First Diffractive Theorem}

We first prove that diffraction exists:
\begin{thm} {\emph{Existence of diffraction:}}
There is an explicit diffractive front corresponding to the inner cone in figure 1 in $u$, the explicit solution of
\be
\left\{ \begin{array}{rcl}
\square u(x,t) + \frac{a}{r^2} u(x,t) &=& 0  \qquad\qquad\qquad (a>0,\; \text{a constant})\\
u_t(x,0)&=&\delta(x-x_0)\\
u(x,0)&=&0
\end{array}\right.
\ee
in $\mathbb{R}^2\times \mathbb{R}$ space-time.
\end{thm}

\begin{proof} We follow the general method used in \cite{cheeger_taylor}. Recall that the solution operator to this perturbed wave equation is $(-\tilde{\Delta})^{-\frac{1}{2}}\sin {t(-\tilde{\Delta})^\frac{1}{2}}$ where $\tilde{\Delta} = \Delta - \frac{a}{r^2}$, $a\in \mathbb{R}^+$, so the fundamental solution is $(-\tilde{\Delta})^{-\frac{1}{2}}\sin {t(-\tilde{\Delta})^\frac{1}{2}} \delta$. We would like an expression for the fundamental solution in terms of polar coordinates to exploit the radial symmetry of the solution, and isolate an expression for the diffractive front in terms of the spectrum of $\tilde{\Delta}$.

We consider $\mathbb{R}^2 = \mathbb{R} \times S^1$ as a cone with base $S^1$, i.e. in polar coordinates $(r,\theta)$. (Our $\theta \in [0,2\pi]$ as usual.) Here, the Laplacian in polar coordinates has the form:
\be
\frac{\partial^2}{\partial r^2}+\frac{1}{r}\frac{\partial}{\partial r}+\frac{1}{r^2}\frac{\partial^2}{\partial\theta^2}
\ee
Define the operator $\nu$ on $S^1$ by $\nu = (\frac{\partial^2}{\partial\theta^2}+a)^\frac{1}{2}$ and $\nu_j = (\mu_j + a)^\frac{1}{2}$ where $\{\mu_j\}$ is the spectrum for $\frac{\partial^2}{\partial\theta^2}$ with eigenfunctions $\varphi(\theta)$. The base $S^1\cong\mathbb{R}/2\pi\mathbb{Z}$, and we get $\nu_n = (n^2 + a)^\frac{1}{2}$ for $n\in \mathbb{Z}$ and $\varphi(\theta) = e^{in\theta}$. Separating variables, we consider the action of $\tilde{\Delta}$ on $g(r,\theta) = \sum_{ij} g_i(r)\varphi_j(\theta)$ namely:
\be\label{radiallap}
\tilde{\Delta} g(r,\theta) = \sum_j-\left(\left[\frac{\partial^2}{\partial r^2}+\frac{1}{r}\frac{\partial}{\partial r}-\left(\frac{\nu_j^2}{r^2}\right)\right] g_j(r)\right)\varphi_j(\theta)
\ee

The form of the radial part strongly suggests \textbf{Bessel's equation}:
\be
\left[\frac{\partial^2}{\partial z^2}+\frac{1}{z}\frac{\partial}{\partial z}-\frac{\nu^2}{z^2}\right] J_{\nu}(z) = -J_{\nu}(z)
\ee
which are solved by the \textbf{Bessel Functions}:
\be
J_\nu(z) = \frac{(z/2)^\nu}{\Gamma(1/2)\Gamma(\nu+1/2)}\int_{-1}^{1}(1-t^2)^{\nu-1/2}e^{izt}dt
\ee

This suggests making use of convolution with the Bessel Functions in a spectral transformation known as the \textbf{Hankel Transform}:
\be
H_\nu(g)(\lambda) = \int_{\mathbb{R}^{+}} g(r) J_\nu(\lambda r) r dr
\ee
for which we know the following proposition:
\begin{prop}
For $\nu \ge 0$ $H_\nu$ extends uniquely from $C_{0}^{\infty} \left(\mathbb{R}^{+}\right)$ to
\be
H_\nu : L^2(\mathbb{R}^{+},r dr)\rightarrow L^2(\mathbb{R}^{+},\lambda d\lambda)
\ee
and for each $g\in L^2(\mathbb{R}^{+},r dr)$:
\be
H_\nu \circ H_\nu g = g
\ee
\end{prop}

(See \cite{cheeger_taylor} for a proof.)

Letting $L_\mu = \left[\frac{\partial^2}{\partial r^2}+\frac{1}{r}\frac{\partial}{\partial r}-\left(\frac{\mu^2}{r^2}\right)\right]$, the radial part of the operator we have $L_{\mu}(J_{\nu}(\lambda r)) = -\gamma^2 J_{\nu}(\lambda r)$ that:
\bea
H_\nu(L_\nu g) &=& \int_{0}^{\infty} L_\nu(J_\nu(\lambda r)) g r dr\\
&=& -\lambda^2 \int_{0}^{\infty} J_\nu(\lambda r) g r dr\\
&=& -\lambda^2 H_\nu(g)
\eea
and this isometry carries the radial portion of $\tilde{\Delta}$ into multiplication by $-\lambda^2$.

Combining this with the Fourier Transform on the spherical variables gives as the Schwarz kernel for the solution:
\be
f(-\tilde{\Delta})(r_1,r_2,\theta_1,\theta_2) = \sum_n e^{in(\theta_1-\theta_2)}\int_0^\infty f(\lambda^2)J_{\nu_n}(\lambda r_1) J_{\nu_n}(\lambda r_2) \lambda d\lambda
\ee
we can write the Schwartz Kernel for our operator $\nu$ for a single Fourier mode $\nu_n$ as:
\be
f(-\tilde{\Delta}) (r_1,r_2,\nu_n) = \int_{0}^{\infty} f(\lambda^2) J_{\nu_n}(\lambda r_1) J_{\nu_n}(\lambda r_2) \lambda \, d\lambda
\ee
Apply this to our actual fundamental solution, $\sqrt{\lambda}\sin (t\sqrt{\lambda}) = \lim_{\epsilon \rightarrow 0}\mathrm{Im}\left(\frac {e^{-(\epsilon + it)\lambda}}{\lambda}\right)$, and then use the \textbf{Lipschitz Hankel Integral} (see \cite{cheeger_taylor}) to obtain the following equality:
\be
\int_{0}^{\infty} e^{-t\lambda}J_{\nu_n}(r_1\lambda)J_{\nu_n}(r_2\lambda)\,d\lambda = \frac{1}{\pi} (r_1 r_2)^{-1/2} Q_{\nu-\frac{1}{2}}\left( \frac{r_1^2+r_2^2+t^2}{2r_1 r_2} \right)
\ee
where $ Q_{\nu-\frac{1}{2}} $ is a \textbf{Legendre function} of second kind.

We take the real solution as the limit of the meromorphic function $ Q_{\nu-\frac{1}{2}}(Z) $ as the $Z$ variable approaches the real line. We will see that $Z>1$ corresponds to the interior of region I, $-1<Z<1$ the interior of region II, and $Z<-1$ the interior region III. Approaching the poles at $Z=-1,1$ from different sides will correspond to the boundary of the three regions defined by the cones in our diagram and we will show there is a non-zero difference along the entirety of the boundary between regions II and III.

To see this, \cite{Lebedev} gives that $Q_{\nu-\frac{1}{2}}$ has the following integral representation:
\be
Q_{\nu-\frac{1}{2}}(Z) = \int_{\cosh^{-1}Z}^{\infty} \frac{e^{s\nu}}{(2\cosh (s) - 2 Z)^{1/2} } \, ds
\ee
The Legendre function is analytic when $Z\notin \mathbb{C} \backslash (-\infty, 1]$; there is a branch cut on the real line and analytic continuation allows us to extend both sides of the equation to the complement of the cut by shifting our contour of integration on the right hand side. As $Z\rightarrow 1$ or $Z\rightarrow -1$, the integral becomes singular, and the directions in which $Z$ approaches these singularities will account for the behavior of the solution on different regions of the lightcone.

Substituting $Z=\cosh(\eta)$ yields:
\be
Q_{\nu-\frac{1}{2}}(\cosh (\eta)) = \int_{\eta}^{\infty} \frac{e^{s\nu}}{(2\cosh (s) - 2 \cosh (\eta))^{1/2} } \, ds
\ee
and note that the cut on $\mathbb{C}$ will change under the $\cosh^{-1}$ transformation in the imaginary line. We have $\cosh \eta \in [-1, 1] \Leftrightarrow \cos i\eta \in [-1, 1] \Leftrightarrow i\eta \in [0,-\pi] \Leftrightarrow \eta \in [0,i\pi]$, where correspondences are $Z=-1$ to $\eta = \pi i$, and $Z=1$ to $\eta = 0$. So the cut becomes $[0,\pi i] \cup [\pi i, \infty + \pi i]$. The upper half plane is transformed into the rectangle enclose by the cut and real line.

We then have three regions to consider on $\mathbb{R}$: I: $\cosh(\eta)\in (1,\infty)$; II: $\cosh(\eta)\in (-1,1)$; III: $\cosh(\eta)\in(-\infty,-1)$ for $Q_{\nu-\frac{1}{2}}(\cosh (\eta))$. The contours taken are the following:

\begin{figure}[htp]
\begin{center}
\fbox{\parbox[c][2.1in][l]{2.4in} %top, 2in deep, v-centered, 4in wide
{\leftline{\includegraphics[scale=0.5]{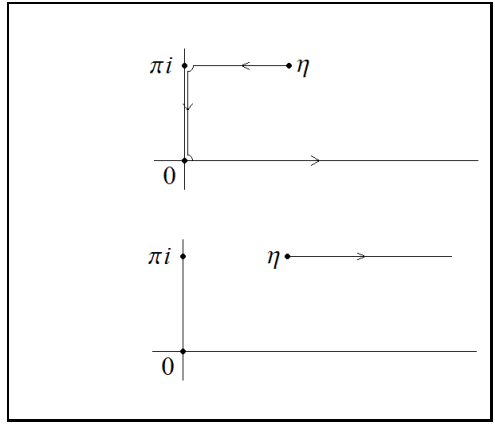}}}
}
\end{center}
\caption{Equivalent contours for case III, we choose the former for its similarity to the contour from case II as we approach $\pi i $ along $[\pi i, \pi i+ \infty]$.}
\label{fig:space}
\end{figure}

In I, we have that the integral is on a purely real contour, and since we are concerned with only the imaginary part of $Q_{\nu-\frac{1}{2}}(\cosh (\eta))$, we have $K(r_1,r_2,\nu)=0$.

In II,  we take the contour (see above II) on the cut. As above, the integral disappears when the contour is on the real line, leaving only the portion from the imaginary segment in $[0,i\pi]$. We translate this into a real integral on a segment between $[0,\pi]$. Taking the imaginary part of the result noting the identity $\cosh(z) = \cos(iz)$, and that an additional $i$ comes from swapping the order of the denominator gives yields:
\be
K(r_1,r_2,\nu_n)=
\frac{1}{\pi}(r_1 r_2)^{-1/2} \int_{0}^{\cos^{-1}\left(\frac{r_1^2+r_2^2-t^2}{2r_1 r_2}\right)}
\frac{\cos\nu_n s}{(t^2-r_1^2-r_2^2+2 r_1 r_2 \cos s)^{1/2}}\, ds
\ee
In III, we have several options for contours that go to $\infty$ without crossing the cut that are equivalent because of analyticity within the rectangle which represents the upper half plane under the $\cosh^{-1}$ transform. We choose our contour as an extension of the contour in II to facilitate calculating the difference between the two (see figure). We add to the portion on $[0,i\pi]$ an additional segment from $\eta$ to $i\pi$ and hence an additional term which we again shift into a real integral where $\beta = \cosh^{-1}\left(\frac{r_1^2+r_2^2-t^2}{2r_1 r_2}\right)$:
\bea
K(r_1,r_2,\nu_n)=
\frac{1}{\pi}(r_1 r_2)^{-1/2}\int_{0}^{\pi}
\frac{\cos\nu_n s}{(t^2-r_1^2-r_2^2+2 r_1 r_2 \cos s)^{1/2}}\, ds \\
-\frac{1}{\pi}(r_1 r_2)^{-1/2}\sin(\pi \nu_n) \int_{0}^{\beta}
\frac{e^{-s\nu_n}}{(2\cosh\beta - 2\cosh s)^{1/2}}\, ds
\eea

\begin{figure}[htp]
\begin{center}
\fbox{\parbox[c][2.2in][l]{2.4in} %top, 2in deep, v-centered, 4in wide
{\leftline{\includegraphics[scale=0.5]{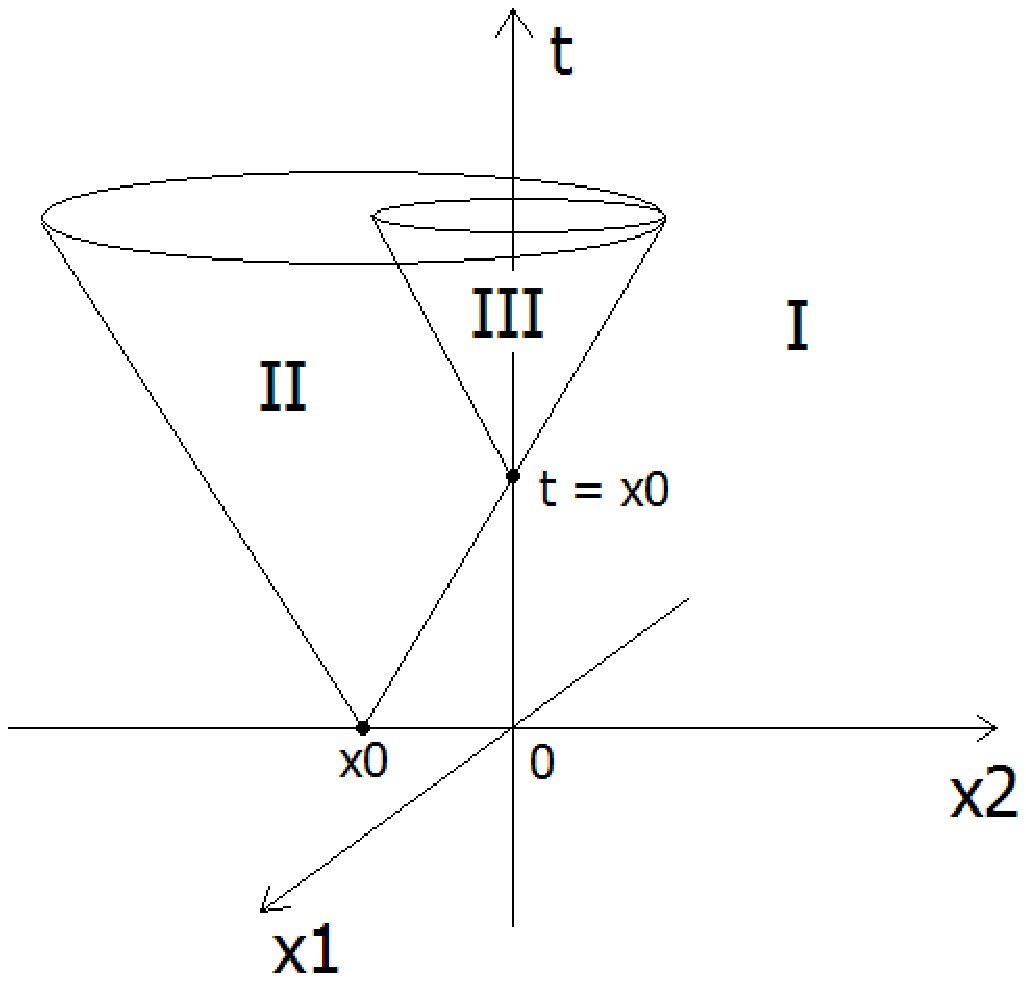}}}
}
\end{center}
\caption{Main front and diffractive front}
%\label{fig:space}
\end{figure}

In the integral for case I, we have $\cosh \eta > 1$ so $\cosh^{-1}\left(\frac{r_1^2+r_2^2-t^2}{2r_1 r_2}\right)>0$ which implies $0 < t < |r_1-r_2|$. We are outside the primary lightcone in case I; the fact that our solution is $0$ here is consistent with finite speed of propagation.

In the integral for II, we have "$-1 < \cosh \eta < 1 $" which translates on our contour into $0<\cos^{-1}\left(\frac{r_1^2+r_2^2-t^2}{2r_1 r_2}\right)<\pi$ and hence $|r_1-r_2|<t<r_1+r_2$. We are inside the primary lightcone but outside of the inner lightcone in case II.

Finally for case III, "$\cosh < -1$" which really means we approach the cut from above. We have $\cosh^{-1}\left(\frac{t^2-r_1^2-r_2^2}{2r_1 r_2}\right)>0$ and $t>r_1+r_2$. This means we are inside the inner lightcone.

Letting $\eta$ approach the boundary of case II and III from opposite directions, that is $r_1\searrow t-r_2$ and $r_1\nearrow t-r_2$ respectively, we see that the solutions in case II and III differ by:
\be
-\frac{1}{\pi}(r_1 r_2)^{-1/2} \sin(\pi \nu_n) \int_{0}^{\beta}
\frac{e^{-s\nu_n}}{(2\cosh\beta - 2\cosh s)^{1/2}}\, ds
\ee

If we take a second order approximation for $\cosh$ near $0$, the limit as we approach the boundary cone of the the integral above is:
\bea
\lim_{\beta\rightarrow 0} \int_{0}^{\beta} \frac{e^{-s\nu_n}}{(2 \cosh \beta - 2 \cosh s)^{1/2}}\,ds &=& \lim_{\beta\rightarrow 0} \int_{0}^{\beta} \frac{ds}{(\beta^2-s^2)^{1/2}}\,ds\\
&=& \lim_{\beta\rightarrow 0} \left.\sin^{-1}\left(\frac{s}{\beta}\right)\right|_{0}^{\beta}\\
&=& \frac{\pi}{2}
\eea
and so the solution exhibits a difference of:
\be
-\frac{1}{2}(r_1 r_2)^{-1/2} \sin(\pi \nu_n) 
\ee
on the $n$-th mode of the Fourier decompsition of $\varphi(\theta)$.

Recall the spectrum of our base, $\nu_n = (n^2 + a)^\frac{1}{2}, n\in\mathbb{Z}$. This means if we have $a \neq 2nm+m^2, m\in \mathbb{Z}$, then $\sin \pi \nu_n \neq 0$. By uniqueness of the spectrum, our solution along the inner cone is not $0$.

We need to eliminate the further possibility of the solution being a collection of singularities at discrete points instead of a legitimate front spread across the diffractive cone. The Schwarz representation theorem tells us that this undesirable case can only come from having a sum of Dirac delta functions and its derivatives $\sum_k c_k \delta^{(k)}$, which has the spectrum $\sum_{k>0} c_k n^k$. Our solution has the spectrum $C\sin (\pi\sqrt{n^2 + a})$ for $C$ some constant; the expression grows asymptotically of order $1$ so it can only possibly match a sum Dirac delta functions without derivatives. However, this has the spectrum $\sum_{k>0} c_k$ which is invariant in $n$ and cannot match all the values of $C \sin(\pi\sqrt{n^2 + a})$ over $n \in \mathbb{N}$ for any values of $a \neq 0$ since it becomes a-periodic because of the irrational values that $\sqrt{n^2 + a}$ takes. Hence the fundamental solution on the diffractive cone is a legitimate, non-zero front by the uniqueness of spectrum for distributions. \end{proof}

\section{Second Diffractive Theorem}

By the first theorem, diffraction indeed occurs when a singularity strikes the origin; we would now like to show theorems in more general contexts which restrict where singularities may propagate.

\subsection{First Formulation}

Recall that by the celebrated Duistermaat-H\"ormander theorem, we have away from $r=0$:

\begin{thm} \label{D-Horm}\emph{(Duistermaat-H\"ormander)}. Let $\square+\frac{a}{r^2} u = 0$. WF$(u)$ is a union of maximally extended null-bicharacteristics of the vector field $H_{\sigma(\square+\frac{a}{r^2})}=2\tau \partial_{t}- \sum 2\xi_j \partial_{x_j}$.
\end{thm}

This says the singularities of $\square + \frac{a}{r^2}$ propagate along unit speed geodesics that fan out from the singular concentration at $x_0$ of our initial condition. (When $r \rightarrow 0$ however, we can't propagate past the large singularity.) This theorem is secretly a statement about the restriction of the propagation of smoothness in wave operator solutions along null-bicharacteristics of the Hamilton vector field. In the potential free case, it in turn restricts the propagation of singularities.

This means that the single geodesic striking the origin in our first diffractive theorem is the sole cause of the entire shower of singularities in the diffractive front.  Thus, in formulating an analogous propagation of smoothness theorem as Duistermaat-H\"ormander for the wave operator with potentia, we cannot localize as finely as we could before since there is no hope of restricting the incoming smoothness or the out going smoothness in anything less than an entire cone.

On the other hand, energy methods allow us to gain some generality in the potential:

\begin{thm} {\emph{Propagation of smoothness:}} 
On $\mathbb{R}^n$ for $n\ge 3$ and $\lambda(n) = \frac{n-2}{2}$, let $u\in \tilde{\mathcal{D}}$ be a solution of
\be
\square u(x,t) + \frac{f(r, \theta)}{r^2} u(x,t) = 0
\ee
where $f(r,\theta) \in \mathcal{C}^\infty(\mathbb{R}^n / \left\{0\right\})$, $f > -\lambda(n)^2$ locally, and $\|f\|_{\infty,\emph{loc}}<\infty$.

If there are no singularities of $u$ going into the point $(0,t)$ originating from this point's backward cone of influence, then there are no singularities coming out of its forward diffractive cone as seen in figure \ref{smoothin_smoothout}.
\end{thm}

In the potential, $f$ does not necessarily have a form $f_1(r)f_2(\theta)$ that separates $r$ from $\theta$, so we can't separate variables as before to find an explicit solution and resort to energy methods. In particular, we will use microlocal energy estimates since boundedness and smoothness of $f$ allow commutator arguments to be made in the style of Duistermaat-H\"ormander and borrow many of the finer techniques used in \cite{melrose_wunsch_vasy}. We are not getting as finely localized a result as we could on our usual Euclidean space, but this result will imply that we don't have spontaneous emission of singularities from the origin due to the potential or trapping of incoming singularities with a delayed release. This is an analogous microlocal version of the local smoothing estimates that appear in \cite{Burq_et}.

Our solutions all belong to a domain which on the spatial variables we pick to be the Friedrichs extension for the quadratic form associated with our operator $\Delta +\frac{f(r,\theta)}{r^2}$:
\be
Q(u):= \int_{\mathbb{R}^n} |\nabla u(x)|^2 + \frac {f(r,\theta)}{r^2} |u(x)|^2 dx
\ee
We will have a more thorough discussion of domains after we define the b-Calculus and its associated b-Wavefront set as well as b-Sobolev spaces. Ultimately, we will want to define all of these objects with respect to our domain, which we will prove to be equivalent to b-Sobolev Spaces using Hardy's Inequality.

We have suppressed the microlocal details of our thereom in its first formulation above. Indeed, our statement refers to incoming and outgoing singularities, which is best formulated in terms of the wavefront set for our solution. This is analogous to the following formulation of Duistermaat-H\"ormander where $c(t)$ is a parametrization of our null-bicharacteristic:

\begin{flushleft}
\textbf{Theorem \ref{D-Horm}'} \emph{Under the hypotheses of \ref{D-Horm} and WLOG letting our point of interest be $t=0$, if $\emph{WF}(u) \cap \{c(t)|\ t<-\epsilon\} = \emptyset$ for some $\epsilon > 0$ then $c(0) \notin \emph{WF}(u)$}.
\end{flushleft}

Due to the prescence of a radial point in the vector field, it will turn out that the best propagating variable to use will be the momemtum variable $\xi$ corresponding to $r$. However, along constant speed geodesics in Euclidean space, this variable makes an abrupt jump as we cross $0$ from negative to positive for any geodesic striking the origin. It is preferable to use the rescaled $r\partial_r$ to generate our vector field and correspond this to the momentum $\xi$. This yields a gentle propagation through the origin with the momentum vanishing to $0$ as it approaches the origin and then gradually increasing in speed as it leaves.

The presence of $r\partial_r$ motivates the use of the \textbf{b-Pseudodifferential Calculus} when formulating our commutant and a \textbf{blowup} at the origin in order to accomodate these operators. These can be thought of as the pseudo-differential calculus generalized from the vector fields generated by $r\partial_r, \partial_\theta, \partial_t$.

\subsection{Blowup and b-Pseudodifferential Calculus}

\begin{figure}[htp]
\begin{center}
\fbox{\parbox[c][2.2in][l]{2.4in} %top, 2in deep, v-centered, 4in wide
{\leftline{\includegraphics[scale=0.5]{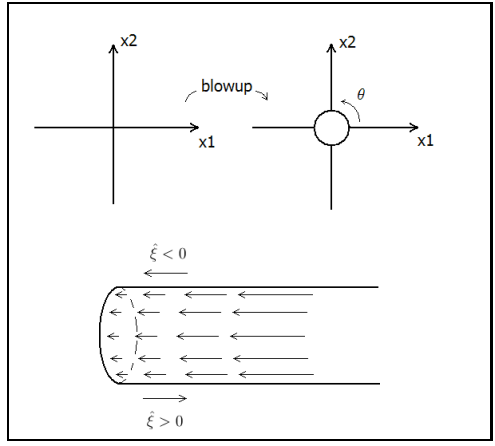}}}
}
\end{center}
\caption{Blowup with circle at ${r=0}$, rescaled $r\partial_r$ vector field, and propagating variable.}
%\label{fig:space}
\end{figure}

Here, the result of our \textbf{blow-up} on $\mathbb{R}^n$ is denoted $[\mathbb{R}^n,0]\cong\mathbb{R}^+ \times \mathbb{S}^{n-1}$; we are continuing polar coordinates from $\mathbb{R}^n-\{0\}$ in a one-to-one way to the origin by joining a sphere. Letting $r$ be our boundary defining function and $\theta_i$ represent spherical variables, the b- vector fields are generated by $r\partial_r, \partial_{\theta_{i}}, \partial_t$; these vectors are tangent to the boundary sphere of our blown up space. In particular, note that a flow in the $r\partial_r$ direction will slow down as we approach the origin (see figure). The bi-filtered *-algebra of pseudodifferential operators associated to these vector fields is known as the b-calculus.

The \textbf{b-Calculus} on a manifold with boundary $M$, whose formalization is due to Melrose (see \cite{melrose}) is a microlocal generalization of vector fields that are tangent to our blown up boundary, spanned in $\mathcal{C}^\infty$ by $r\partial_r$, $\partial_t$ and $\partial_{z_i}$ which correspond to the symbols $\xi$, $\tau$, and $\zeta_i$. The dual to such tangent vectors is \textbf{b-cotangent bundle} $^bT^*M$. The dual of $r\partial_r$ is $\frac{dr}{r}$, $\partial_{\theta_{i}}$ is $d\theta_i$, and $\partial_t$ is $dt$. We denote the coordinates with respect to this basis by $\xi$, $\zeta$, and $\tau$ thus making our canonical one-form $\xi \frac{dr}{r} + \zeta_i d\theta^i + \tau dt$. The \textbf{co-sphere bundle} is $^bS^*M = (^bT^*M-\{0\})/\mathbb{R}^+$, and it is here where our characteristic set lives and where our propagation takes place. 

The b-Calculus of pseudo differential operators is denoted $\Psi^{m,l}_b(M)$. It is a bi-filtered ($m$ for differential order, $l$ for weight order) *-algebra of operators acting on $\dot{\mathcal{C}}^\infty$ and $\mathcal{C}^{-\infty}$, smooth functions that vanish to infinite order at the boundary and their corresponding dual respectively. They satisfy:

%\begin{itemize}
	\begin{description}
		\item{\textbf{(I)}} $\text{Diff}^m_b(M)\subset \Psi^{m,0}_b$,
		\item{\textbf{(II)}} for $x$ a boundary defining function of $M$, we have $x^l\in \Psi^{0,l}_b(M)$ and $x^l\Psi^{m,0}_b(M) = \Psi^{m,l}_b(M)$,
		\item{\textbf{(III)}} there is a \textbf{principal symbol map} is a bi-filtered *-algebra homomorphism which maps to weighted polyhomogeneous symbols
			\be
			^b\sigma_{m,l}:\Psi^{m,l}_b(M)\rightarrow x^l\ S^m_{hom}(^bT^*M \backslash {0})
			\ee
		\item{\textbf{(IV)}} the principal symbol sequence is exact
			\be
			0\rightarrow \Psi^{m-1,l}_b(M) \hookrightarrow \Psi^{m,l}_b(M) \rightarrow x^l S^m_{hom}(^b T^*M \backslash {0})
			\ee
		\item{\textbf{(V)}} for $A\in \Psi^{m,l}_b(M), B\in \Psi^{m',l'}_b(M)$ with principal symbols $a$ and $b$ of respective orders, $[A,B]\in \Psi^{m+m'-1,l+l'}(M)$ obeys
			\be
			^b\sigma_{m+m'-1,l+l'}([A,B])= \frac{1}{i}\{a,b\}
			\ee
		\item{\textbf{(VI)}} all $A\in \Psi^{0,0}$ on $\dot{\mathcal{C}}^{\infty}(M)$ extend by continuity to a bounded $L^2$ operator.
	\end{description}
%\end{itemize}

Notice that the original weight of $r^{-l}$ remains in the residual terms quotiented away in the principal symbol map. For our commutator argument which iterates on successively on the level of the principal symbol, this will prove to be too crude. We will eventually prove lemmas that allow us to explicitly place weights on specific differential operators which we factor out at each successive step.

We say $A\in \Psi^{m,l}_b$ is \textbf{elliptic} at $p \in ^bS^*M$ if its principal symbol has an inverse near $p$ in $r^{-l}\mathcal{C}(^bS^*M):=r^{-l} S^{-m}_{hom}(^bT^*M)/S^{-m-1}_{hom}(^bT^*M)$. We call these points $Ell_b(A)$. There is a corresponding concept of microsupport as well, given by the essential support of an operator's (left) symbol, denoted WF'$_b(A)$. This obeys

%\begin{itemize}
\begin{description}
\item{\textbf{(I)}} \emph{WF}'$_b(AB) \subset$ \emph{WF}'$_b (A)\cap$ \emph{WF}'$_b (B)$
\item{\textbf{(II)}} For $A\in \Psi^{m,l}_b(M)$ and $p\in Ell_b(A)$, there is a parametrix $Q\in\Psi^{-m,-l}(M)$ such that
\be
p\notin \emph{WF}'_b(QA-I)\cup\emph{WF}'_b(AQ-I)
\ee
\item{\textbf{(III)}} For $A \in \Psi^{m,l}_b(M)$ and \emph{WF}'$_b(A)=\emptyset$, we have $A\in \Psi^{-\infty,l}_b(M)$
\end{description}
%\end{itemize}
Note that $\emph{WF}_b$ corresponds to the usual $\emph{WF}$ away from the boundary.

Letting $L^2_b(M) = \left\{ u \in L^2_{loc}(X^\circ); \int_X |u|^2 \frac{dr}{r} d\sigma \right\}$, we have corresponding b-Sobolev spaces defined by 
\begin{defi}
$u\in H^{m,l}_b(M) \Leftrightarrow \Psi^{m,-l}_b(M) u \subset L^2_b(M)$.
\end{defi}

We can define \textbf{b-Wavefront} set in $^bT^*M$ for a distribution in an analogous way too:

\begin{defi}
$\emph{WF}^{m,l}_b (u) = \left\{p\in ^bT^*M| \exists A \in \Psi^{0,0}_b \emph{elliptic at p}, Au \in H^{m,l}_b(M)\right\}^c$
\end{defi}

and the usual properties hold:
%\begin{itemize}
\begin{description}
\item{\textbf{(I)}}	\emph{WF}$^{m,l}_b(u)$ is closed and conic,
\item{\textbf{(II)}}	\emph{WF}$^{m,l}_b(u) \cap T^*M ^\circ= \emph{WF}^{m,l}(u) $
\item{\textbf{(III)}}	$\bigcap_{m,l}H^{m,l}_b(M) = \dot{\mathcal{C}}^\infty(M), \bigcup_{m,l}H^{m,l}_b(M) = \mathcal{C}^{-\infty}(M)$
\item{\textbf{(IV)}}	$A \in \Psi^{m',l'}_b(M)$ maps $A: H^{m,l}_b(M) \rightarrow H_b^{m-m',l+l'}(M)$
\item{\textbf{(V)}}	for $m \leq m'$ $\emph{WF}^{m,l}_b(u) \subset \emph{WF}^{m',l}_b(u)$
\item{\textbf{(VI)}}	for $u\in H^{-\infty,l'}_b(u)$ and $A\in\Psi^{k,l}(u)$
\be
\emph{WF}^{m,l+l'}_b(Au) \subset \emph{WF}_b(A) \cap \emph{WF}^{m+k,l'}_b(u),
\ee
and
\be
\emph{WF}^{m+k,l'}_b(u) \backslash \emph{WF}^{m,l+l'}_b(Au) \subset (Ell_b A)^c.
\ee
\item{\textbf{(VII)}}	$H^{m,l}_b\hookrightarrow H^{m',l'}_b$ is compact if $m<m'$ and $l<l'$. In particular, operators in $\Psi^{m,l}_b$ are compact if $m<0$, $l>0$.
\end{description}
%\end{itemize}

\subsection{Domains}

We now establish a domain $\tilde{\mathcal{D}}$ in which all of our arguments will take place. Let $\mathcal{D}$ denote the Friedrichs form domain of $\Delta+\frac{f(r,\theta)}{r^2}$, so the closure of
$\dot{\mathcal{C}}^\infty(M)$ with respect to the quadratic form:
\be
Q(u):= \int_{\mathbb{R}^n} |\nabla u(x)|^2 + \frac {f(r,\theta)}{r^2} |u(x)|^2 dx
\ee
We then define $\tilde{\mathcal{D}}$  by adding in $D_t$:
norm $\|u\|_{\mathcal{D}}^2=\|D_t u\|^2+\|\nabla u\|^2_{L^2}+\|\frac{1}{r}u\|^2_{L^2}$.

The $s/2$-th power of the operator $\Delta$ and $D_t$ will define $\mathcal{D}_s$ analogously.

\begin{lem}\label{domain_lemma}
For $u\in\dot{\mathcal{C}}^\infty(M)$ and $n\ge 3$, we have
\be
\frac{1}{C}\|u\|^2_{\tilde{\mathcal{D}}} \leq \langle \frac{f}{r^2}u,u\rangle+\|\frac{1}{r}D_\theta u\|^2 + \|D_r u\|^2 + \|D_t u\|^2 \leq C \|u\|^2_{\tilde{\mathcal{D}}}
\ee
\end{lem}

This requires \textbf{Hardy's Inequality} which accounts for the $n\ge 3$ dimensional condition in our theorem.
\begin{lem} \emph{Hardy's Inequality} (from \cite{hardy}):For $u \in H^1(R^n)$, $n\ge 3$, $u(0)=0$ we have:
\be
\int_{\mathbb{R}^n}\left|\frac{u(x)}{|x|}\right|^2 dx \leq \left( \frac{2}{n-2} \right)^2 \int_{\mathbb{R}^n}|\partial_r u(x)|^2 dx
\ee
where the constant $\frac{1}{\lambda^2}$ is the best possible.
\end {lem}
\emph{Proof:} 
See \cite{hardy} for the proof using integration by parts. $\Box$

We have as a special case from proposition 1 of \cite{Burq_et} that under our assumptions of dimension above on $f$ in our potential which guarantee that for the Laplace-Beltrami operator on the sphere $\displaystyle{\not}\Delta$, $\displaystyle{\not}\Delta + f(r,\theta) + \lambda(n)^2$ ($\lambda(n) := \frac{n-2}{2}$) is a positive operator on every sphere, we have:
\begin{prop} {\emph{Equivalence of Norms:}} \cite{Burq_et}\label{equiv_of_norms} There are constants $c_1$ and $c_2$ such that:
\be
c_1 \|\nabla u(x)\|^2 \leq Q(u) \leq c_2 \|\nabla u(x)\|^2
\ee
\end{prop}

So the domain of our space under powers of $P:= \Delta + \frac{f(r,\theta)}{r^2}$ is equivalent to the standard homogeneous Sobolev norm based on the powers of $\Delta$.

\begin{proof} To show \ref{equiv_of_norms}, we can pick $c_2 = \left(1+\frac{\|f\|_\infty}{|\lambda|^2}\right)$ since from Hardy's Inequality:
\be
Q(u)\le \int |\nabla u|^2 + \frac{\|f\|_\infty}{|x|^2}|u|^2 \le \left(1+\frac{\|f\|_\infty}{|\lambda|^2}\right) \|\nabla u\|^2_{L^2}
\ee
To find $c_1$ (see statement of proposition), we integrate radially. Because of the positivity of $\displaystyle{\not}\Delta + f(r,\theta) + \lambda^2$ on every sphere, there is some $\delta > 0$ such that:
\bea
Q(u) &=& \int_0^{\infty}\int_{|x|=r} |\partial_r u|^2 + \frac{1}{r^2}|\nabla_\theta u|^2 + \frac{f(r,\theta)}{r^2}|u|^2 d\sigma dr \\
&\ge& \int_0^{\infty} \frac{1}{r^2}\int_{|x|=r} |\nabla_\theta u|^2 + (f(r,\theta)+\lambda^2)|u|^2 d\sigma dr\\
&\ge& \int_{0}^{\infty}\frac{\delta^2}{r^2} \int_{|x|=r} |u|^2 d\sigma dr \\
&=& \delta^2 \|\frac{1}{r}u\|^2_{L^2}
\eea
Letting $c_1=\frac{\delta^2}{\delta^2 + \|f\|_{\infty}}$ we have using the above:
\bea
Q(u) - c_1 \|\nabla u\|^2_{L^2} &=& (1-c_1)Q(u) + c_1 \int_{\mathbb{R}^n} \frac{f(r,\theta)}{r^2}|u|^2 dx\\
&\ge&\int_{\mathbb{R}^n} (-c_1 f(r,\theta)+(1-c_1)\delta^2)\frac{|u|^2}{|x|^2}dx
\ge 0
\eea
\end{proof}

\begin{proof} In \ref{domain_lemma}, clearly the only term that arises that is not included in the definition of the domain is $\|\frac{1}{r} u\|$ which by Hardy's Inequality is bounded by $\|D_r u\|$ and hence by the domain norm.
\end{proof}

Note that when extending this domain, we require an additional $\|u\|_{L^2}$ term. Hardy's inequality will hold with this extra term if we decompose $u$ by an appropriate cutoff function, supported at the origin with sufficiently small mass, i.e. $u = \psi u + (1-\psi) u$.

We also define a local version of the norm:
\be
\|u\|_{\cdot, loc}=\|\phi u\|_{\cdot}
\ee
for $\phi \in \mathcal{C}_c^\infty(M)$ fixed to contain the region of interest.

b-Pseudodifferential operators interact nicely with this domain, in particular, 0-th order operators are still bounded and commute as expected with b-Differential operators which can be seen by using expanding the domain definition and using the usual theorems. This allows us to define b-Sobolev spaces with respect to this domain:

\begin{defi}
For $m\ge0$, let $H^{m}_{\tilde{\mathcal{D}},b,c}$ be the subspace of $\tilde{\mathcal{D}}$ of compactly supported $u$ with $Au \in \tilde{\mathcal{D}}$ for $A$ elliptic over $supp (u)$. Then $H^{m}_{\tilde{\mathcal{D}},b,loc}$ is the subspace of $\tilde{\mathcal{D}}$ of $u$ such that for any $\phi \in \mathcal{C}_c(M)$, $\phi u \in H^{m}_{\tilde{\mathcal{D}},b,c}$.

For $m<0$, let $H^{m}_{\tilde{\mathcal{D}},b,c}$ are $u\in\mathcal{C}^{-\infty}(M)$ of the form $u=u_1+Au_2$ with $u_1,u_2\in\tilde{\mathcal{D}}_{loc}$ and $A\in\Psi_b^{-m}(M)$. The norm is:
\be
\|u\|_{H^{m}_{\tilde{\mathcal{D}},b,c}}=\inf\left\{\|u_1\|_{\tilde{\mathcal{D}}}+\|u_2\|_{\tilde{\mathcal{D}}}:u=u_1+Au_2\right\}
\ee
with $H^{m}_{\tilde{\mathcal{D}},b,loc}$ the space of $u \in \mathcal{C}^{-\infty}(M)$ such that $\phi u \in H^{m}_{\tilde{\mathcal{D}},b,c}$ for $\phi \in \mathcal{C}^{\infty}_c(M)$.
\end{defi}

Note that these definitions are independent of the $A$ chosen. We also require wavefront set with respect to these Sobolev spaces:
\begin{defi}
Let $u\in H^{m}_{\tilde{\mathcal{D}},b,loc}$ for some $s$. For $q \in ^bT^*M / \left\{0\right\}$, $q\notin \emph{WF}^m_{b,\tilde{\mathcal{D}}(u)}$ if there exists an $A \in \Psi^m_b(M)$ such that $A$ is elliptic near $q$ and $Au\in \tilde{\mathcal{D}}$.
\end{defi}

\begin{lem}
$\tilde{\mathcal{D}} = rH^1_b(M).$ 
\end{lem}
\begin{proof} Use the definition of the norm of $rH^1_b(M)$ and the norm equivalence lemma. \end{proof}

\begin{cor}
Operators in $\Psi_b^0$ are bounded on $\tilde{\mathcal{D}}$ and on $H^m_{b,\tilde{\mathcal{D}}}$ for all $m$.
\end{cor}
\begin{proof}
Use the standard square root argument with respect to the Sobolev domain and note that this is equivalent to our new domain.
\end{proof}

\subsection{Second Formulation}

We can now reformulate the theorem on the blownup space in terms of b-wavefront set.

The co-sphere bundle is a more appropriate place to state theorems about the propagation of $\emph{WF}(u)$ since they are already conic by definition and so propagation in radial directions is hence only a matter of presence at all in a cone. To get an appropriate set of coordinates on the co-sphere bundle we pick $\tau$ which is a degree 1 homogeneous function that does not vanish on our characteristic set, so we can renormalize fibers using it by letting $\hat{\xi}:= \frac{\xi}{\tau}, \hat{\zeta}:=\frac{\zeta}{\tau}, \hat{\tau}:=\pm 1$ although only the first two variables will be used.

The notion of \textbf{b-elliptic regularity} gives us that the solution to our equation is regular outside the characteristic set. This restriction corresponds to the restriction to unit speed geodesics from our original formulation. Note that $\square := D_t^* D_t - (D_r^* D_r - \frac{1}{r^2} \Delta_{\theta} + \frac{f(r,\theta)}{r^2})$, where $D_r^* = D_r + \frac{n-1}{r}$ has a b-principal symbol of $\tau^2 - \frac{\xi^2}{r^2} - \frac{|\zeta|^2}{r^2}$. Multiplying by $r^2$ and dividing by $\tau$ gives the following characteristic set:
\bea
\Sigma &:=& \{(r,z_i,t,\xi,\zeta_i, \tau)|\sigma_{2}(r^2\Box+f(r,\theta))=0\} \\
&=&\{(r, z_i, t, \xi, \zeta, \tau) | r^2-\hat{\xi}^2-|\hat{\zeta}|_k^2=0\}
\eea
where $k$ denotes the length induced by the metric on the base $\mathbb{S}^{n-1}$, so our solutions have $\emph{WF}_{b,\tilde{\mathcal{D}}}(u)\in \Sigma$. On the interior of our blowup (away from the new boundary), we have $\emph{WF}_b(u)\cap ^bT^*_{X^o}X = \pi(\emph{WF}_{b,\tilde{\mathcal{D}}}(u)\cap T^*_{X^o}X)$ where $\pi$ is the blowdown map, so b-regularity translates into standard $C^\infty$ regularity and vice-versa. The Duistermaat-H\"ormander theorem gives that $\emph{WF}_{b,\tilde{\mathcal{D}}}(u)$ is the usual maximally extended family of bicharacteristics inside $\Sigma$ on the interior of our blown up $\mathbb{R}^n$.

It is useful to know an exact for for the Hamilton vector field here.

\begin{lem}\cite{melrose_wunsch_vasy}\label{commutator_lemma1}If $A\in \Psi_b^m(M)$ then the Hamilton vector field $H_a$ of $a = \sigma_{b,m}(A)$ defined initially on the interior of the blown up cotangent bundle $T^*M^{\circ}$ extends to the boundary as an element of $\nu_b(^bT^*M)$ in coordinates $(x,t,z,\xi,\tau,\zeta)$ as:
\be
H_a=(\partial_{\xi}a)r\partial_r+(\partial_{\tau} a)\partial_t+(\partial_{\zeta_i}a)\partial_{\theta_i}-(r\partial_r a)\partial_{\xi}-(\partial_t a)\partial_{\tau}-(\partial_{\theta_j} a)\partial_{\zeta_j}
\ee
\end{lem}
\begin{proof} Use the definition in coordinates and the b-projection map. \end{proof}

So written as a b-operator, our Hamilton vector field is:
\be
H_{\sigma(\Box+\frac{f(r,\theta)}{r^2})}= 2\tau \partial_t - \frac{2\xi}{r^2} r\partial_r - \frac{k^{ij}\zeta_j}{r^2}\partial_{\theta_i} +\frac{\zeta_i \partial_{\theta_i}(k^{ij})\zeta_j}{r^2}\partial_{\zeta_i}
-\frac{2(\xi^2+|\zeta|^2)}{r^2}\partial_\xi 
\ee

Furthermore:

\begin{lem} Angular momentum $\zeta = 0$ for any bicharacteristics that strike the origin under this Hamilton vector field.
\end{lem} 
\begin{proof}
To see this, lets analyze the flow for:
\be
H_{\sigma(\Box+\frac{f(r,\theta)}{r^2})}= 2\tau \partial_t - \frac{2\xi}{r^2} r\partial_r 
-\frac{\kappa^{ij}\zeta_j}{r^2}\partial_{\theta_i} +\frac{\zeta_i \partial_{\theta_i}(k^{ij})\zeta_j}{r^2}\partial_{\zeta_i}
-\frac{2(\xi^2+|\zeta|_k^2)}{r^2}\partial_\xi
\ee
Call the propagating variable $s$ and $'$ denote push-forward by $\partial_s$. Our system is:
\be
\left\{ \begin{array}{rcl}
t'&=&\tau \\
r'&=&-\frac{\xi}{r}\\
\theta_i'&=&\frac{\kappa^{ij}\zeta_j}{2r^2}\\
\tau'&=& 0\\
\xi'&=&-\frac{(\xi^2+|\zeta|^2)}{r^2}\\
\zeta_j'&=&\frac{\zeta_i \partial_{\theta_i}(k^{ij})\zeta_j}{r^2}
\end{array}\right.
\ee
We reparametrize flow with respect to the a parameter such that $ds = r^2 dr$, and rescaled this becomes:
\be
\left\{ \begin{array}{rcl}
t'&=&r^2\tau \\
r'&=&-r\xi\\
\theta_i'&=&-\frac{\kappa^{ij}\zeta_j}{2}\\
\tau'&=& 0\\
\xi'&=&-\xi^2+|\zeta|^2\\
\zeta_j'&=&\zeta_i \partial_{\theta_i}(k^{ij})\zeta_j
\end{array}\right.
\ee

When $|\zeta|_k^2>0$, the coupled system in $r$ and $\xi$ produce:
\begin{equation}
\left\{\begin{array}{rcl}
\log(r(s)) &=& \int \xi(s) ds\\
\xi(s) &=& \sqrt{|\zeta|_k^2} \tan(\sqrt{|\zeta|_k^2}(s-c_1))\\
\end{array}\right.
\end{equation}
which means
\begin{equation}
\log(r(s)) = \log(|\sec(\sqrt{|\zeta|_k^2}(s-c_1))|) - c_2\\
\end{equation}
Since $\sec$ is never 0, $r\neq 0$ along this flow.
\end{proof}

There is a radial point at $r=0, \zeta=0$ for where null-bicharacteristics strike the origin. This prevents us from using a diffeomorphism which orients the vector field in some canonical direction and propagating through the origin as in the standard proof of Duistermaa-H\"ormander.

We pick $-\hat{\xi} = -\frac{\xi}{\tau}$ as our propagating variable, and note that $H_{\sigma(\Box+\frac{f(r,\theta)}{r^2})}(-\hat{\xi})$ is decreasing. This fact is crucial to constructing an appropriate commutant for our commutator argument, and it corresponds to the size of $\frac{2(\xi^2+|\zeta|^2)}{r^2}\partial_\xi$ relative to the other terms. By picking the support of the commutant appropriately, we can show that this is the dominant term in our positive commutator argument. Also note that we're in the characteristic set along bicharacteristics with no angular momentum so $r^2 = \hat{\xi}^2$. Notice that $\hat{\xi}$ vanishes as we approach the origin instead of flipping and as we change sign, it accounts for outgoing bicharacteristics with the opposite momentum, and the absence of wavefront set where $\hat{\xi} > 0$ will imply that there are no singularities on outgoing null-bicharacteristics. Restated again, our theorem becomes:
\begin{thm}\label{mainpropagation2} Let $u$ in our domain be a solution of $\left(\Box + \frac{f(r,\theta)}{r^2}\right)u = 0$ and over $q_0=(t_0,\tau_0)$ let $U$ be a neighborhood in $S^*_b(\mathbb{R}\times[\mathbb{R}^n,0])$ around the set $Q = \{r=0, t=t_0, \tau_0=\pm 1, \hat{\xi}= 0, \theta \in \mathbb{S}^{n-1}, \hat{\zeta}=0\}$. Let $\tilde{U} := U \cap \Sigma$, $\Sigma$ the characteristic set defined before. In the region where $\tau > 0$:
\be
\tilde{U}\cap\{r > 0, -\hat{\xi} < 0\}\cap \emph{WF}(u) = \emptyset \Longrightarrow Q \cap \emph{WF}_{b,\tilde{\mathcal{D}}} (u) = \emptyset.
\ee
\end{thm}

\subsection{Proof of Second Diffractive Theorem}
\subsubsection{overview}
We will inductively prove the reformulated theorem above inductively using the relative wavefront set. Because we are assuming all our solutions belong to $\tilde{\mathcal{D}}$, they must be in some Sobolev space to begin with. Assuming (without loss of generality) that $Q\cap \emph{WF}^0_b(u)=\emptyset$, we want to show that $Q\cap \emph{WF}^{1/2}_b(u)=\emptyset$. This would then guarantee inclusion in all orders of Sobolev regularity, showing the propagation of $\mathcal{C}^\infty$ regularity.

\begin{lem}\label{existence_of_commutant}
There exists $A\in \Psi^0_b(M)$ such that:
\bea
i\langle[\Box+\frac{f(r,\theta)}{r^2},A^*A]u,u\rangle = -\langle B^*Bu,u\rangle - \sum_j\langle G_j^*G_ju,u\rangle\\ + \langle E_1u,u\rangle + \langle E_2u,u\rangle + \langle Ru,u \rangle
\eea
In particular, $\Vert Bu\Vert$ detects order $1/2$ wavefront set around $Q$ and is bounded since every other term can be shown to be bounded. Therefore, we gain an additional $1/2$ order of regularity around $Q$. Note that the norm used here is relative to the domain $\tilde{mathcal{D}}$ defined earlier.
\end{lem}

On the left hand side: assuming there is sufficient regularity in $u$ for the left hand side to make sense, the left hand side will evaluate to $0$ by using the self adjointness of our operators and that $\left(\Box+\frac{f(r,\theta)}{r^2}\right)u=0$.

On the right hand side: we will construct this commutant using microlocal cutoffs, (i.e. one that cuts of in both spatial and phase variables). By its design, the commutant will expand into a number of terms on the right hand side. The first term $-\langle B^*Bu,u\rangle$ is the main term which detects $\emph{WF}^{1/2}_b(u)$ around the set $Q$; it is our goal to bound this main term and hence show an additional $1/2$ order of regularity. This can be done by showing that all the other terms are either bounded or positive with the same sign as our main term. The second term $- \sum_i\langle G_i^*G_iu,u\rangle$ is a term with the same fixed sign as the main term so it helps our estimate. This term will be adjustable in size from how we construct $A$, and we absorb lower order terms from differentiation with non-propagating variables in the Hamilton vector field into this term. The third term $\langle E_1u,u\rangle$ will be supported on the incoming co-tangent neighborhood and is bounded by our wavefront hypothesis $\tilde{U}\cap\{r > 0, -\hat{\xi} < 0\}\cap \emph{WF}(u) = \emptyset$. The fourth term $\langle E_1u,u\rangle$ is supported off of the characteristic set $\Sigma$ and is bounded by elliptic regularity. The final term $\langle Ru,u \rangle$ is meant to be lower order, and ultimately bounded by inductive hypothesis.

There will be technical caveats to consider regarding the final remainder term $\langle Ru,u \rangle$ which prevent us from working on the b-principal symbol level exclusively. However, suppressing these matters for the time being, we can provide a sketch of the lemma and construct $A$ on the level of its b-principal symbol.

First, let $ \chi (x)$ be a smooth bump function with smooth square-root such that:
\be
\chi(x)=\left\{ \begin{array}{rcl}
1 && [-1,1] \\
0&&(-\infty, -2]\cup[2, \infty)\\
\end{array}\right.
\ee
where we call $\chi'(x) := \phi_1^2(x) - \phi_2^2(x)$, each $\phi$ smooth, and $\tilde{\chi(x)}$ be a smoothed Heaviside function
\be
\tilde{\chi}(x)=\left\{ \begin{array}{rcl}
1 && (1,\infty] \\
0&&(-\infty, 0]\\
\end{array}\right.
\ee
for which $\chi'(x) := \phi_3^2(x)$, $\phi_3$ also smooth.

In selecting $\chi(x)$, we require its derivative be the pair of bump functions with opposite signs $\chi'(x) = \phi_1^2(x) - \phi_2^2(x)$, with $\phi_1(x)$ supported on $[-2,-1]$ and $\phi_2(x)$ supported on $[1,2]$. Similarly, let $\tilde{\chi}(x)$ and its derivative have a smooth square root supported on $[0,1]$

\begin{figure}[htp]
\begin{center}
\fbox{\parbox[c][4.6in][l]{3.6in} %top, 2in deep, v-centered, 4in wide
{\leftline{\includegraphics[scale=0.5]{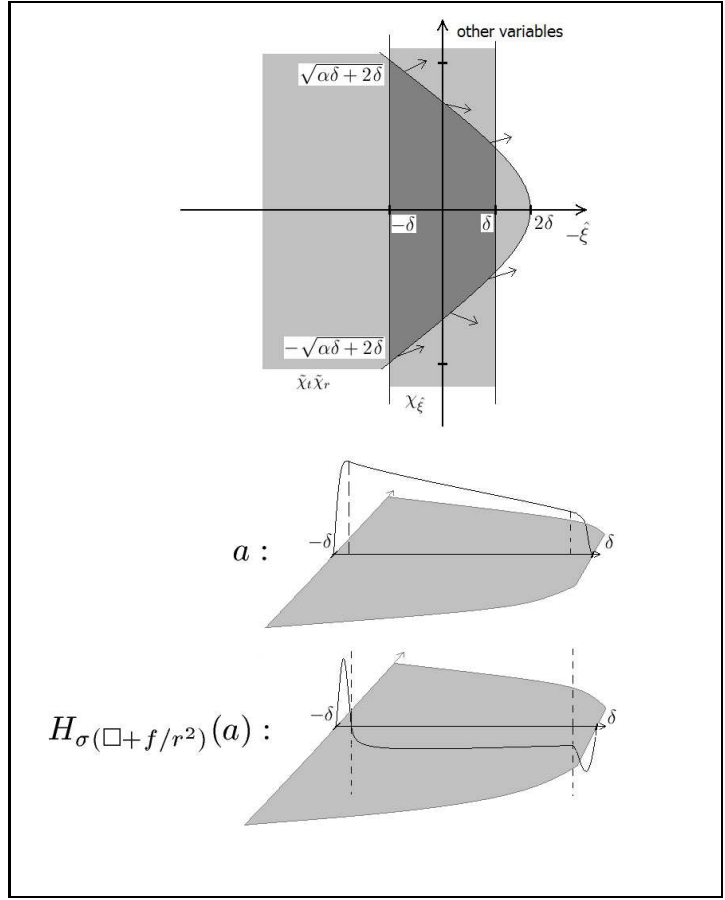}}}
}
\end{center}
\caption{\textbf{Top:} tapering support ensures differentiation against a falling edge. (See end of chapter note for more detail about behavior in the $r$ variable.) \textbf{Middle:} profile of commutant. \textbf{Bottom:} profile of commutant under differentiation by Hamilton vector field; the bump on the left is supported where we have wavefront assumption, the non-zero middle portion detects wavefront set at the point of interest, the bump on the right has the same sign due as the middle portion. The like signed bump extends along the curved edges of the support (not pictured here) thanks to differentiation occurring only against the falling edge of the support.}
\end{figure}

Taking the above remarks into account, we use the following construction for our commutant's b-principal symbol for some $\tau_0 > 0$ to ensure uniform boundedness of $\tau$ from 0:
\bea
a(r,z,t,\xi,\zeta,\tau) := e^{C\hat{\xi}}\chi\left(\frac{\hat{\xi}}{\delta}\right)\tilde{\chi}\left(-r^2+\alpha\hat{\xi}+2\delta\right)\\
\tilde{\chi}\left(-(t-t_0)^2+\alpha\hat{\xi}+2\delta\right)\tilde{\chi}(\tau-\tau_0)\chi\left(\frac{r^2-\hat{\xi}^2-|\hat{\zeta}|^2}{\delta}\right)
\eea
with A some constant. We abbreviate this by:
\be
a := e^{C\hat{\xi}}\chi_{\hat{\xi}}\tilde{\chi_{r}}\tilde{\chi_{t}}\tilde{\chi}_\tau\chi_{\Sigma}
\ee
This symbol will localize in any given co-sphere neighborhood of our set of interest. It is clear that $\hat{\xi}$ is localized by $\delta$. The variables $r$ and $t-t_0$ are localized by the interaction of their parabolic cutoffs with $\xi$'s vertical cutoffs. The maximum width of the parabola cut off by the vertical is at most $\sqrt{\alpha\delta+2\delta}$, which we can control for any given $\alpha$.

The first term, $e^{C\hat{\xi}}$ for $C>0$ a constant, decreases in $-\hat{\xi}$ to ensure a main term with a consistent negative sign under differentiation by $-\partial_\xi$ near our point of interest, and hence creates a nonzero symbol for detecting $\emph{WF}$ at our point of interest in our main term $B$. (Note that being in the characteristic set gives $\hat{\xi}^2+|\hat{\zeta}|_k^2=r^2$ so on bicharacteristics that strike the origin, we localize in $r$ as a result. )

The second term, $\chi_{\hat{\xi}}$, localizes in $\hat{\xi}$, and under differentiation we get a term with support in the wavefront hypothesis region (left vertical edge of diagram) as well as one with the same sign as our main term (right vertical edge of diagram). The former is the symbol for $E_1$ and the latter contributes to the $G$.

The terms, $\tilde{\chi}_{r}\tilde{\chi}_{t}$ cut off in the slow variables, and the terms $\alpha\hat{\xi}$ introduce a parabolic tapering with respect to propagation in $-\hat{\xi}$. This ensures that the Hamilton vector field points out of our support, and thus creates a bump function which is consistent in sign when hitting falling edge. These terms will make a small contribution to $G$, but will be dominated by the contribution from $-\hat{\xi}$.

The term $\tilde{\chi}(\tau)$ keeps $\tau > 0$. Finally, $\chi_{\Sigma}$, cuts off fiber variables to ensure that we have a symbol; the $\delta$ parameter is meant to adjust the support of $\chi$ so that it is in some small tubular neighborhood of the characteristic set. Since this term is constant in a neighborhood near the characteristic set, differentiating it yields bumps supported off of the characteristic set whose quantizations are bounded by elliptic regularity. This term also keeps $r>0$ when $-\hat{\xi}$ is negative and enables us to use the $\emph{WF}$ hypothesis on $E_1$.

We can directly see these terms appear by a direct application of $H_{\Box+\frac{f}{r^2}}$ to the symbol we constructed:

\be
H_{\Box+\frac{f}{r^2}} (a)= -b^2 - g + e_1 + e_2
\ee

where:

\bea
b^2:=&& -2\left[\frac{ \xi^2+|\zeta|^2 }{r^2}\right] \left[\frac{A}{\tau} e^{C\hat{\xi}}\chi_{\hat{\xi}}\tilde{\chi}_r\tilde{\chi}_t\tilde{\chi}_\tau\chi_{\Sigma}\right]\\
g:=&&\\
&&-e^{C\hat{\xi}} \left[ \phi_1^2\left(\frac{\hat{\xi}}{\delta}\right) \left[\frac{2}{\delta\tau}\left[\frac{ \xi^2+|\zeta|^2 }{r^2}\right] \right]\tilde{\chi}_r\tilde{\chi}_t\tilde{\chi}_\tau\chi_{\Sigma} \right.\\
&& +\chi_{\hat{\xi}} \phi_3^2\left( r^2+\alpha\hat{\xi}+2\delta\right) \left[ -2\xi + 2 \frac{\alpha}{\tau}\left[\frac{\xi^2+|\zeta|^2}{r^2} \right]\right]\tilde{\chi}_t\tilde{\chi}_\tau\chi_{\Sigma}\\
&& +\chi_{\hat{\xi}}\tilde{\chi}_r \phi_3^2\left( (t-t_0)^2+\alpha\hat{\xi}+2\delta\right)\\
&& \left.\qquad\qquad\qquad \left[ 2(t-t_0)\tau + 2\frac{\alpha}{\tau} \left[\frac{\xi^2+|\zeta|^2}{r^2}\right] \right] \tilde{\chi}_\tau\chi_{\Sigma}\right]\\
e_1:=&& e^{C\hat{\xi}} \phi_2^2\left(\frac{\hat{\xi}}{\delta}\right) \left[\frac{2}{\delta\tau}\left[\frac{ \xi^2+|\zeta|^2 }{r^2}\right] \right]\tilde{\chi}_r\tilde{\chi}_t\tilde{\chi}_\tau\chi_{\Sigma}\\
e_2:=&& e^{C\hat{\xi}} \tilde{\chi}_{\hat{\xi}} \tilde{\chi}_r\tilde{\chi}_t\tilde{\chi}_\tau H_{\Box+\frac{f}{r^2}}(\chi_{\Sigma})
\eea

Note that for $g$, since $\tau > 0$, $\tau^2 - \frac{\xi^2+|\zeta|_k^2}{r^2} < \delta$, and $\xi$ as well as $t-t_0$ are cut off from being too big by $\tilde{\chi}_r$ and $\tilde{\chi}_t$ respectively (in conjunction with $\tilde{\chi}_{\hat{\xi}}$), we can choose $\alpha$ to be as large as we'd like and force the entire term to be negative.

The final term $e_2$ is clearly supported off of $\Sigma$ since $\chi_{\Sigma}$ is constant in a tubular neighborhood of $\Sigma$ and hence any derivatives will only create terms off of $\Sigma$.

It is tempting to now quantize in the b-Calculus and finish the argument stated in lemma \ref{existence_of_commutant}, however $\langle Ru,u\rangle$ requires additional care. This is because our domain involves b-differential operators of the form $D_t$, $\frac{1}{r}rD_r$, and $\frac{1}{r}D_{\theta_i}$, but the lower order error term arising due to the quantization only a priori obey $R\in \frac{1}{r^2}\Psi^{-1}_b(M)$ because of a deficiency in the filtration from the principal symbol map on the b-calculus. Using the operations of the calculus against a pair of operators will presume the ``worse'' weight of the two in the result, and this means that it is possible non domain bounded terms such as $\frac{1}{r}D_t Op(\cdot)$ are present in $R$.

There is a more technical argument on the sub-principal level we will now present which hinges on factorizing out the domain related $D_t$, $D_r$, and $\frac{1}{r}D_{\theta_i}$ terms at each step of the iterative argument. An elliptic regularity lemma from in \cite{melrose_wunsch_vasy} lets us relate and these terms. In particular, we end up with a bound on a single term involving $D_t$ which is elliptic since $|\tau| > 0$ and hence detects $\emph{WF}$ around the origin.

The remainder also includes the inverse square potential adjoined to $\square$, which can be absorbed into $\frac{1}{r}rD_r$ by an application of Hardy's Inequality as we shall see in the full proof.

\subsubsection{Setup For the Factorized of Main Propagation Theorem \ref{mainpropagation2}}

Our factorized argument will be similar to what was presented in the overview by iterating on differential order, but now we explicitly factor out $D_t$, $\frac{1}{r}rD_r$, and $\frac{1}{r}D_{\theta_i}$ at iterative step to ensure a domain bounded term for all lower order remainders. We denote these good differential operators by $Q_i$.

In the end, we want to investigate the behavior of
\be
\langle -i[A^*A,\Box+\frac{f(r,\theta)}{r^2}] u, u \rangle
\ee
iteratively while maintaining a factorization by $Q_i$'s. This involves expanding the operators on the right hand side by Leibnitz rule and investigating the properties of each resulting term. We further expand each term on the principal symbol level and classify these terms into the categories by the template:
\be
H_{\Box+\frac{f}{r^2}} (a)= -b^2 - g + e_1 + e_2
\ee
seen as seen in the previous section. A lemma will allow us to conclude that the terms corresponding to $G$ have a consistently ``good sign'' (i.e. are always positive). The terms $E_1$ and $E_2$ are bounded by wavefront hypothesis and elliptic regularity respectively. Finally the $R$ term is always domain bounded by the way we iterated this argument only after factorizing out terms which are well behaved in the domain. This shows that the elliptic $B$ terms which will finally be $\langle D_t T B u, D_t T B u \rangle$ are bounded and hence detect a higher order of smoothness than before at the origin.

We choose $A\in \Psi^0_b(M)$ with principal symbol $a = \sigma_{0,0}(A)$ as before:

\bea
a(r,z,t,\xi,\zeta,\tau) := e^{C\hat{\xi}}\chi\left(\frac{\hat{\xi}}{\delta}\right)\tilde{\chi}\left(-r^2+\alpha\hat{\xi}+2\delta\right)\\
\tilde{\chi}\left(-(t-t_0)^2+\alpha\hat{\xi}+2\delta\right)\tilde{\chi}(\tau-\tau_0)\chi\left(\frac{r^2-\hat{\xi}^2-|\hat{\zeta}|^2}{\delta}\right)
\eea

The b-quantization of this symbol requires us to use charts on $\mathbb{S}^{n-1}$. First we take $\{\phi_i(x)\}$ a partition of unity supported in these coordinate charts. Let $\{\chi_i(x)\}$ be smooth cutoffs equal to 1 on a neighborhood of supp$(\phi_i)$ and still supported in the coordinate charts. Let $\rho \in C^\infty_c ((-1/2, 1/2))$ identically 1 near 0. Now in local cooridnates, we can define the b-quantization using projective coordiantes in $r$:
\be
A_i := (2\pi)^{-n}\int e^{i\left(\xi \frac{r-r'}{r'}+\zeta\cdot(z-z')\right)+\tau(t-t')}\rho\left(\frac{r-r'}{r}\right)a(r,t,\xi,\tau,\zeta)\phi_i(\theta)d\xi d\tau d\zeta \sigma
\ee
with $\sigma$ some choice of right density. (Note that our $a$ is independent of $\theta$. We've also just blown up in spatial variables, not time.) Define globally
\be
Op_{trv}(a) = \sum_i \chi_i A_i \chi_i
\ee

Recall that $\square = D_t^*D_t - D_r^*D_r-\frac{1}{r^2}\Delta_{\mathbb{S}^{n-1}}$ where $D_r^* = D_r + \frac{n-1}{r}$ and $\Delta_{\mathbb{S}^{n-1}}$ is the spherical Laplacian which in coordinates is of the form $\Delta_{\mathbb{S}^{n-1}} = \sum_{ij} \frac{1}{\sqrt{|k|}}D_{\theta_i}k^{ij}\sqrt{|k|} D_{\theta_j}$ for $k^{ij}\in C^\infty$ on some chart and $|k|$ the determinant of $k^{ij}$.

We first investigate how our commutant commutes with any of our $Q^i$'s. For $A:=Op_{trv}(a)$ as before, we now consider $A^*A$ (which is very similar on the principal symbol level to $A$) for self-adjointness considerations, and by Leibnitz rule:
\bea
\lefteqn{-i[A^*A,\Box+\frac{f(r,\theta)}{r^2}]=} \\
&& \label{commutant_line1}i[A^*A,D_t^*]D_t + D_t^*i[A^*A,D_t]\\
&&\label{commutant_line2}-i[A^*A,D_r^*]D_r - D_r^*i[A^*A,D_r]\\
&&\label{commutant_line3}-i[A^*A,\frac{1}{r^2}]\left(\sum_{ij} \frac{1}{\sqrt{|k|}}D_{\theta_i}k^{ij}\sqrt{|k|} D_{\theta_j}+f(r,\theta)\right)\\
&&\label{commutant_line4}-\frac{1}{r^2}\left(i\lbrack A^*A,\sum_{ij} \frac{1}{\sqrt{|k|}}D_{\theta_i}k^{ij}\sqrt{|k|} D_{\theta_j}\rbrack+i\lbrack A^*A,f(r,\theta)\rbrack\right)
\eea

We will need some lemmas about the behavior of the $[A^*A,Q^i]$ terms present above.

The following lemma about the commutators involving $D_r$ from \cite{melrose_wunsch_vasy} allow us to improve on the order of the weight more than just what the b-calculus' filtration tells us. Note that as a b-operator $D_r = \frac{1}{r} rD_r$ and so a priori in the b-Calculus filtration $\lbrack D_r, A\rbrack \in \frac{1}{r}\Psi^{2m-1}_b(M)$.

\begin{lem}\label{commutator_lemma2}\cite{melrose_wunsch_vasy}If $A\in\Psi_b^m(M)$, there exist $B\in\Psi_b^m(M)$, $C\in\Psi_b^{m-1}(M)$ depending continuously on $A$ such that
\be
[A, D_r] = B+CD_r
\ee
with $\sigma(B) = -i\partial_r(\sigma(A))$, $\sigma(C) = -i\partial_{\xi}(\sigma(A))$.
\end{lem}
\begin{proof} See \cite{melrose_wunsch_vasy} for the proof involving normal operators.
\end{proof}

It is crucial that $B$ and $C$ on the right hand side are weightless operators. The proof of this lemma makes use of the fact that $\lbrack rD_r, A\rbrack$ has an additional order of weight regularity than the b-Calculus would a priori give us. This lemma is shown by using normal operators which have their coefficients involving the $r$ variable are frozen at $r=0$. Please refer to \cite{melrose_wunsch_vasy} for further details.

The next lemma deals with the behavior of $[A,D_\theta]$. Our commutant $A$ has principal symbol that is constant in the spherical variables $\theta_i$; such symbols are called \textbf{basic} symbols so this means that $D_{\theta_i}(a)=0$ on fiber variables on each chart. We are still left with some non characteristic terms though:

\begin{lem}\cite{melrose_wunsch_vasy}\label{commutator_lemma3}If $A \in \Psi_b^m(M)$ where $A=Op(\tilde{a}\psi\left(\frac{r^2-\hat{\xi}^2+|\hat{\zeta}|^2}{\delta}\right))$ where $\tilde{a}$ is independent of $\theta$ and $\psi$ cuts off near $\Sigma$ and
\be
\lbrack A, D_{\theta_i} \rbrack = B \in \Psi_b^m(M)
\ee
obeys $\emph{WF}(B)\cap \Sigma = \emptyset$
\end{lem}

\begin{proof}
Recall the quantization on each chart:
\be
A_i := (2\pi)^{-n}\int e^{i\left(\xi \frac{r-r'}{r'}+\zeta\cdot(z-z')\right)+\tau(t-t')}\rho\left(\frac{r-r'}{r}\right)a(r,t,\xi,\tau,\zeta)\phi_i(z)d\xi d\tau d\zeta \sigma
\ee
and
\be
Op_{trv}(a) = \sum_i \chi_i A_i \chi_i
\ee

Note the independence of $a$ from $\theta$ and $\zeta$, and so such operators commute with $D_\theta$ at the cost of terms $\lbrack \chi, D_\theta \rbrack$ which are smoothing terms because the Schwartz Kernel is smooth. These will not affect our lemma as they are residual.

On each coordinate patch near the diagonal, we are able to represent the quantization under a change of variable in $\theta$ as:
\be
A_i := (2\pi)^{-n}\int e^{i\left(\xi \frac{r-r'}{r'}+\zeta\cdot(z-z')\right)+\tau(t-t')}\rho\left(\frac{r-r'}{r}\right)\tilde{a}(r,t,\xi)\psi_i(r,\theta,\xi,\zeta)d\xi d\zeta d\sigma
\ee
Note our $a$ has the form $a:=\tilde{a}\chi_\Sigma$ and that $\psi_i = \phi_i \chi_\Sigma$ here has a change of variables under $z$. The sums of the derivatives on $\psi$ will cancel under summation. The only possible contributions which come from $\lbrack A, D_\theta \rbrack$ must be from $d(\psi)$ where $\chi_\Sigma$ is non-zero, which must be non-characteristic.
\end{proof}

A final lemma gives the behavior of $[A,\frac{1}{r^2}]$

\begin{lem} If $A\in\Psi_b^m(M)$, then:
\be
[A^*A, \frac{1}{r^2}] = C
\ee
where $C\in \frac{1}{r^2}\Psi_b^{2m-1}$, $\sigma(C) = -\frac{1}{r^2}2i\partial_\xi(\sigma(A))$.
\end{lem}

\begin{proof}
We use the same lemma as in the proof of (\ref{commutator_lemma2}), which specifies that $\lbrack rD_r, A\rbrack$ has an additional order of weight regularity than the b-Calculus would a priori give us. Here:
\be
\lbrack A, r^2 \frac{1}{r^2} \rbrack = \lbrack A, r^2 \rbrack \frac {1}{r^2} + r^2\lbrack A, \frac {1}{r^2} \rbrack 
\ee
and so
\be
\lbrack A, \frac {1}{r^2}\rbrack = \frac {1}{r^2}\lbrack A, r^2\rbrack \frac {1}{r^2}
\ee
where the middle term on the right gains us 2 degree of regularity as before. The principal symbol comes immediately from lemma \ref{commutator_lemma1}.
\end{proof}
\begin{lem}
For the specific $A^*A\in \Psi_b^0$ which we mentioned before, this yields:
\be
\begin{array}{cccc}
& & \text{Principal Symbol} & \text{Order} \\
%\emph{from} (\ref{commutant_line1}) 
& \lbrack A^*A,D_t^*\rbrack \ \& \ \lbrack A^*A,D_t\rbrack &  & \\
 & = L_i & \sigma(L_i) = 2ia\partial_t(a) & \in \Psi_b^0(M) \\
%\emph{from} (\ref{commutant_line2})  
& \lbrack A^*A,D_r^*\rbrack \ \& \ \lbrack A^*A,D_r\rbrack &  &\\
&  = B_i+C_i D_r^{(*)} & \sigma(B_i) = 2ia\partial_r(a) & B_i \in \Psi_b^0(M) \\
& & \sigma(C_i) = 2ia\partial_\xi(a) & C_i \in \Psi_b^{-1}(M) \\
%\emph{from} (\ref{commutant_line3}) 
& \lbrack A^*A,\frac{1}{r^2}\rbrack & & \\
& = \frac{1}{r^2} C_5 & \sigma (C_5) = 4ia\partial_\xi(a) & \in \Psi_b^{-1}(M) \\
%\emph{from}(\ref{commutant_line4}) 
& \lbrack A^*A,Op_\Sigma\rbrack  & & \\
& = E_6 + R& \emph{WF}(E_6)\cap \Sigma = \emptyset \ \text{(basic operator)} &\\
& &\sigma(R) \in \Psi_b^{-1}(M) & \\
%\emph{from}(\ref{commutant_line4}) 
& \lbrack A^*A,f(r,\theta)\rbrack & & \\
& = C_7 + E_7 & \sigma(C_7) = 2ia\partial_{\xi}(a)r\partial_r(f) &   \in \Psi_b^{-1}(M)\\
& & \emph{WF}(E_7)\cap \Sigma = \emptyset &  \in \Psi_b^{-1}(M) \\
\end{array}
\ee
where $Op_\Sigma = \sum_{ij} \frac{1}{\sqrt{|k|}}D_{\theta_i}^*k^{ij}\sqrt{|k|} D_{\theta_j}$.
\end{lem}
\begin{proof}
See the preceding lemmas for the first three rows.

The non-characteristic terms are from $\partial_\zeta$ terms which are supported outside our characteristic set as our symbol is basic. Note also that $k^{ij}$ depend only on $\theta_i$'s and the principal symbol of $\lbrack A^*A, k^{ij}\rbrack$, $\lbrack A^*A, \sqrt{|k|}\rbrack$, and $\lbrack A^*A, \frac{1}{\sqrt{|k|}}\rbrack$ are $2ia\partial_\zeta (a) \partial_\theta (k^{ij})$, $2ia\partial_\zeta (a) \partial_\theta (\sqrt{|k|})$, and \newline $2ia\partial_\zeta (a) \partial_\theta (1/\sqrt{|k|})$ respectively. Since $k_{ij}$ is a metric, its inverse is smooth and each of terms with derivatives in $\theta$ will be a smooth function in $\theta$ (with respect to trivialization on chart coordinates which we have built into our quantization). Our basic symbol then forces the highest order of these terms off the characteristic set.

The final row comes immediately from \ref{commutator_lemma1}.
\end{proof}

Now we can write expressions for these factorized forms whose remainders only pick up additional domain bounded terms. Temporarily renaming $C_3$ and $C_5$ to $\tilde{C}_3$ and $\tilde{C}_5$ reserve them for after we commute, we get:

\bea
\lefteqn{-i\lbrack A^*A, \square+\frac{f(r,\theta)}{r^2}\rbrack=} \\
&& \lbrack L_1 D_t + D_t^* L_2 - \{(B_3 + \tilde{C_3} D_r^*)D_r + D_r^*(B_4 + C_4 D_r^*) + \\ &&\frac{1}{r^2} \tilde{C_5} (\sum_{ij} D_{z_i}^*k^{ij}D_{z_j} - f(r,\theta))+ E_6+R\}\rbrack+\frac{1}{r^2}(C_7+E_7)
\eea
We can regroup to fit our final commutator argument better by letting $E_{3} = \lbrack D_r^*, C_{3} \rbrack \in \Psi_b^{-1}(M)$ and $E_5 = \lbrack D_{z_i}^*, C_5] \in \Psi_b^{-1}(M)$, where the latter operator is supported off the characteristic set and so is absorbable in $E_6$; call $C_{3} = \tilde{C}_{3}+E_{3}$ which we can also reabsorb by using that C is basic. This gives:
\bea
-i\lbrack A^*A, \square+\frac{f(r,\theta)}{r^2}\rbrack &=& \lbrack \label{first_line}L_1 D_t +D_t^* L_2 - B_3 D_r - D_r^* B_4\\
&&\label{potential_line}-(D_r^*(C_3 + C_4) D_r + \sum_{ij} \frac{1}{r}D_{\theta_i}^*C_5k^{ij}\frac{1}{r}D_{\theta_j} \\
&&- \frac{f(r,\theta)}{r}C_5\frac{1}{r}) +\frac{1}{r^2}(C_6 -E_6 + E_7) + R \rbrack
\eea
Here $R$ is a sum of terms from $Q^i\Psi^{-1}_b$ and $(Q^i)^2\Psi^{-2}$ which is lower order and comes from the commutators of our explicitly factorized terms terms above. (In particular, we note that these are bounded with respect to lower order domain norm when paired as $\langle R u, u\rangle$.) When we ultimately pair with our solution $u$, we want to absorb the terms with only first order differential operators in line (\ref{first_line}) into line (\ref{potential_line}). We will see that this is because the principal symbol of the second line will dominate the smaller symbol of the first as in the principal level overview argument earlier.

Let's examine the terms dominant $C_3+C_4$ and $C_5$ present in (\ref{first_line}) and see what contributions their (common) principal symbol $4ia\partial_\xi a$ makes.

Recall our commutant:
\bea
a(r,z,t,\xi,\zeta,\tau) := e^{C\hat{\xi}}\chi\left(\frac{\hat{\xi}}{\delta}\right)\tilde{\chi}\left(-r^2+\alpha\hat{\xi}+2\delta\right)\\
\tilde{\chi}\left(-(t-t_0)^2+\alpha\hat{\xi}+2\delta\right)\tilde{\chi}(\tau-\tau_0)\chi\left(\frac{r^2-\hat{\xi}^2-|\hat{\zeta}|^2}{\delta}\right)
\eea
abbreviated as: $a := e^{C\hat{\xi}}\chi_{\hat{\xi}}\tilde{\chi}_r\tilde{\chi}_t\chi_\tau\chi_{\Sigma}$. First we make a few factorizations to isolate the types of terms we are interested in:
\bea
4ia\partial_{\xi}a &=& 4i\label{mainterm}\frac{A}{\tau} e^{2 A \hat{\xi}} \chi^2_{\hat{\xi}} \tilde{\chi}^2_r \tilde{\chi}^2_t \tilde{\chi}_\tau^2\chi^2_{\Sigma}\\
&&+4i\label{wavefronthypothesis}e^{C \hat{\xi}}\chi_{\hat{\xi}}\left[\phi_{1,\xi}^2-\phi_{2,\xi}^2\right] \left[\frac{1}{\delta \tau}\right]\tilde{\chi}^2_r \tilde{\chi}^2_t \tilde{\chi}_\tau^2 \chi^2_{\Sigma}\\
&&+4i\label{goodsign1}e^{C \hat{\xi}}\chi^2_{\hat{\xi}}\tilde{\chi}_r \phi_{3,r} \left[\frac {\alpha}{\tau}\right]\tilde{\chi}^2_t \tilde{\chi}_\tau^2 \chi^2_{\Sigma}\\
&&+4i\label{goodsign2}e^{C \hat{\xi}}\chi^2_{\hat{\xi}}\tilde{\chi}^2_r \tilde{\chi}_t \phi_{3,t}\left[\frac{\alpha}{\tau}\right] \tilde{\chi}_\tau^2 \chi^2_{\Sigma}\\
&&+4i\label{noncharacteristicterm}e^{C\hat{\xi}}\chi^2_{\hat{\xi}}\tilde{\chi}^2_r\tilde{\chi}^2_t\chi_{\Sigma}\tilde{\chi}_\tau^2\partial_\xi \left(\chi_{\Sigma}\right)
\eea
Recalling the template for the terms we want in our pairing from the lemma in the previous section, this suggests we decompose each into the following:
\be
-iC_\cdot = -T_{-1}^*B^*BT_{-1}-G+E_1+E_2+R
\ee

The term $T_{-1}\in \Psi^{-1}_b(M)$ with $\sigma(T_{-1})=\frac{1}{|\tau|}\chi_\Sigma$ (the $\chi_\Sigma$ makes this a symbol) is elliptic. We will use this term in subsequent sections to shift orders on paired terms.

As before, $B\in\Psi_b^{1/2}(M)$ is the main term quantized from
\be
b:=\sqrt{\tau} A e^{ A \hat{\xi}} \chi_{\hat{\xi}} \tilde{\chi}_r \tilde{\chi}_t \tilde{\chi}_\tau\chi_{\Sigma}
\ee
i.e. $\sigma(B^*B)$ is $-i\tau^2$ times line (\ref{mainterm}) and is elliptic at the boundary $r=0$(from our assumption that $\tau>0$) so that it detects $1/2$ order of extra regularity there.

$G \in \Psi_b^{-1}(M)$ the ``good sign term'' from (\ref{wavefronthypothesis}) (\ref{goodsign1}) and (\ref{goodsign2}). It corresponds only to the terms arising from the $\xi$ derivative, and we define its principal symbol $g$ by:
\bea
\label{good_sign_paired}-g:=&&-e^{2A\hat{\xi}} \chi_\xi \phi_1^2\left(\cdot\right) \left[\frac{1}{\delta\tau} \right]\tilde{\chi}_r^2\tilde{\chi}_t^2\tilde{\chi}_\tau^2\chi^2_{\hat{\xi},\hat{\zeta}} \\
&& -e^{2A\hat{\xi}}\chi_{\hat{\xi}}^2 \tilde{\chi}_r \phi_1^2\left( \cdot \right) \left[ \frac{\alpha}{\tau}\right]\tilde{\chi}_t^2\tilde{\chi}_\tau^2\chi^2_{\hat{\xi},\hat{\zeta}}\\
&& -e^{2A\hat{\xi}}\chi_{\hat{\xi}}^2\tilde{\chi}_r^2 \tilde{\chi}_t \phi_1^2\left( \cdot \right) \left[\frac{\alpha}{\tau} \right]\tilde{\chi}_\tau^2\chi^2_{\hat{\xi},\hat{\zeta}}
\eea

$E_1\in \Psi_b^{-1}(M)$ the ``$\emph{WF}$ hypothesis term'' from (\ref{wavefronthypothesis}),
\be
e_1:= e^{2A\hat{\xi}} \phi_2^2\left(\frac{\hat{\xi}}{\delta}\right) \left[\frac{2}{\delta\tau}\left[\frac{ \xi^2+|\zeta|^2 }{r^2}\right] \right]\tilde{\chi}^2_r\tilde{\chi}^2_t\tilde{\chi}^2_\tau\chi^2_{\Sigma}\\
\ee

$E_2\in \Psi_b^{-1}(M)$ from (\ref{noncharacteristicterm}) the non characteristic term:
\be
e^{C\hat{\xi}}\chi^2_{\hat{\xi}}\tilde{\chi}^2_r\tilde{\chi}^2_t\chi_{\Sigma}\tilde{\chi}_\tau^2\partial_\xi \left(\chi_{\Sigma}\right)
\ee
and $R\in \Psi_b^{-2}(M)$ a domain bounded lower order remainder. (See previous section for a diagram.)

The main summand above involving $C_3+C_4$ and $C_5$ can be rearranged at the cost of commutator terms to form:
\bea
-T_{-1}^*B^*\left(D_r^*D_r+\sum_{ij}\frac{1}{r}D_{\theta_i}k^{ij}\frac{1}{r}D_{\theta_j} - \frac{f(r,\theta)}{r^2}\right)BT_{-1}\\
-G\left(D_r^*D_r+\sum_{ij}\frac{1}{r}D_{\theta_i}k^{ij}\frac{1}{r}D_{\theta_j} - \frac{f(r,\theta)}{r^2}\right)+E_1+E_2+R
\eea
where $\kappa_{ij}$ correspond to the smooth functions that appear in the Laplacian as before. We group together terms with properties corresponding to $E_1$ and $E_2$. We have also grouped resulting commutator terms (which are all of small enough order) into the remainder term, which are now are sums of the remainder terms before and also new terms $R \in Q_i^*Q_i\Psi^2_b(M)$. (These are also bounded by our domain norm) Note that in particular, $B$ is basic (its symbol constant on fiber variables) so commutators with $D_{\theta_i}$ are off the characteristic set and we can combine these into our existing non characteristic set term.

Note that all subsequent terms named $E_1$, $E_2$,$R$ will not necessarily the same expression in every line, they will simply obey the corresponding properties corresponding to terms bounded by wavefront hypothesis, bounded by elliptic regularity (from having support off $\Sigma$), and being a lower order domain bounded remainder term respectively.

Since we want our final bound to be on the term $\langle D_t T Bu, D_t T Bu \rangle$, we want to exchange the differential operators $D_r$ and $\frac{1}{r}D_{\theta_i}$ in line (\ref{first_line}) to $D_t$. This can be done using $\left(\square + \frac{f(r,\theta)}{r^2}\right)u=0$ to help us. As we know the explicit form of the Laplacian $-\left(\square+\frac{f(r,\theta)}{r^2}\right) = D_t^*D_t - D_r^*D_r-\sum_{ij}\frac{1}{r}D_{\theta_i}k^{ij}\frac{1}{r}D_{\theta_j}+\frac{f(r,\theta)}{r^2}$, we can exchange $D_r$ and $\frac{1}{r}D_{\theta}$ terms for $D_t$ by writing the above as:
\bea
-W\left(\square+\frac{f(r,\theta)}{r^2}\right) - B^*T_{-1}^*(D_t^*D_t)T_{-1}B \\
-G\left(D_t^*D_t\right)+E_1+E_2+R
\eea
where $W$ is given by $B^*T_{-1}^*T_{-1}B+G$ where we have used the fact that $B$ is basic and $T_{-1}$ is lower order to absorb commutators from commuting terms around into the $E$ and $R$ terms again.

Likewise, $L_i$ and $B_i$ have symbols:
\bea
2ia\partial_t a &=& 2ie^{2A\hat{\xi}}\chi^2_{\hat{\xi}}\chi^2_{r}\tilde{\chi}_t\phi_{1,t}(\cdot)\left[2(t-t_0\right]\tilde{\chi}_\tau^2\chi^2_{\hat{\xi},\hat{\zeta}}\\
2ia\partial_r a &=& 2ie^{2A\hat{\xi}}\chi^2_{\hat{\xi}}\tilde{\chi}_r\phi_{1,r}(\cdot)\left[2r\right]\chi^2_{t}\tilde{\chi}_\tau^2\chi^2_{\hat{\xi},\hat{\zeta}}
\eea
Note the similarity to the terms which we termed $G_t$ and $G_r$ above, and in final pairing expressions, these terms will be absorbed into the $G$ terms which have good sign and make no contribution against our main commutator.

\subsubsection{Main argument for the Propagation Theorem \ref{mainpropagation2}}

Now for the main inductive argument where we introduce the pairing with $u$, a domain bounded solution. For now assume $s=0$ without loss of generality. (We will shift $s$ with a separate operator at the end which smooths as well.) Because of the restriction of our solution to a domain, we can assume that for some small $s_0$, that at our point $q_0 \notin \emph{WF}^{s_0}_{b,\tilde{\mathcal{D}}}$. Now we assume $\emph{WF}^s_{b,\tilde{\mathcal{D}}} u \cap U = \emptyset$. We will show that $\emph{WF}^{s+1/2}_{b,\tilde{\mathcal{D}}} u \cap \tilde{U} = \emptyset$ for some open neighborhood $\tilde{U}\subset U$ Let's pair the actual commutator expression we will use with our solution, expanding using the expressions we established above:
\bea
\lefteqn{-i\langle\lbrack A^*A, \square + \frac{f(r,\theta)}{r^2}\rbrack u, u \rangle=}\\
&& \langle -W (\Box+\frac{f(r,\theta)}{r^2})u, u\rangle\\
&&-\langle D_tT_{-1}Bu, D_tT_{-1}Bu\rangle \\
&&\label{goodsignline}- \langle G D_t^*D_t u,u \rangle \\
&&\label{absorbed_into_good_sign}+ \langle L_1D_t u, u \rangle + \langle L_2D_t u, u \rangle + \langle B_3 D_r u, u \rangle + \langle B_4 u, D_r u\rangle \\
&&+\langle \frac{1}{r^2}C_7 u, u\rangle\\
&&+ \langle E_1 u, u \rangle + \langle E_2 u, u\rangle + \langle \frac{1}{r^2}E_7 u, u\rangle + \langle R u, u\rangle
\eea

Note that this is only a formal expression as we don't necessarily have sufficiently regularity on $u$ to know that these $L^2$ pairings make sense; we will address this issue with an additional regularizing term and a weak convergence argument at the end of the proof. Here we have absorbed the non-characteristic $E_6$ and $E_7$ term into $E_2$. (The weight of $\frac{1}{r^2}$ will make no difference as after using the cauchy-schwarz followed by Hardy's inequality, we can substitute $\frac{1}{r}$ with $D_r$ plus some multiple of the domain norm of $u$ which is inductively bounded.)

That $G$ be positive is crucial. Using the form of the principal symbol from (\ref{good_sign_paired}), we decompose $G$ into 3 operators with positive principal symbol $G = G_\xi^* G_\xi + G_r^* G_r+ G_t^* G_t$, where $G_\xi$, $G_r$, and $G_t$ and their adjoints are quantized from the square roots of $-i$ times the principal symbols corresponding to the lines starting at (\ref{good_sign_paired}). We will ultimately absorb the $L_iD_t$ and $B_i D_r$ terms from line (\ref{absorbed_into_good_sign}) into the $G_t$ and $G_r$ terms respectively. By showing $L_iD_t$ and $B_i D_r$ terms are smaller in norm than the norm of the corresponding positive $G_tD_t$ and $G_rD_t$ terms respectively in (\ref{goodsignline}), we see that their sum with $G$ is still positive.

The naive, purely principal symbol level argument works for the $D_t$ terms as it doesn't have a b-weight. First consider the terms $\langle L_i D_t u, u\rangle$ which we now absorb into $G_t$. To do this, we consider the difference $\frac{1}{2}D_t^*G_t^* G_tD_t-L_iD_t$ as there are two terms and perform one step of a square root argument to achieve a positive term which helps us as it will have the same sign as $G_t$ plus a residual term with lower order that has no weight and is domain bounded by inductive hypothesis. We get
\be
\langle (D_t^*G_t^* G_t-L_i)D_t u, u\rangle = \langle S^* S u, u\rangle + \langle R u, u\rangle
\ee
where $\sigma(S) = \sqrt{\sigma((G_tD_t^*D_t-L_iD_t))} = \sqrt{\chi_{\hat{\xi}}^2\tilde{\chi}_r^2 \tilde{\chi}_t \phi_1^2\left( \cdot \right) \lbrack 2\alpha  - 2(t-t_0)\rbrack \tilde{\chi}_\tau^2\chi^2_{\hat{\xi},\hat{\zeta}}}$,\\ $R\in D_t\Psi_b^{-1}(M)$. This is a smooth symbol as $\tau > 0$ and we can make $\alpha$ as large as we'd like to force the factor $2\alpha - 2(t-t_0)$ to be positive. Therefore, this term will have the same sign as our main term and the residual term is bounded by hypothesis. The same technique works for the other term involving $D_t$.

We could try to use the same technique for the terms $\langle B_iD_r u, u\rangle$, however as a b- operator, $D_r = \frac{1}{r}rD_r$ and we would impose an arbitrary weight on our remainder which is not attached to the specific differential operators mentioned above, which was the whole point of factorizing $Q^i$ at each step. Instead, we show an estimate based on a lemma that allows us to estimate domain derivative in terms of each other, in particular $D_r$ and $\frac{1}{r}D_{\theta_i}$ can be bounded interms of $D_t$.

We factorize $B_i$ in terms of $G$ which becomes:
\be
B_i= G_r^* T_{1}^* R_i T_{1}G_r + \tilde{R}
\ee
where $\sigma(R_i)=\frac{2 r}{\alpha \tau}\chi_{\Sigma}$, $\sigma(T_1)=\langle \tau \rangle\chi_{\Sigma}$ and $\tilde{R}\in\Psi_b^{-1}(M)$ is residual. Clearly the supremum of the symbol of $\sigma(R_1) = \frac{2 r}{\alpha \tau}\chi_{\Sigma}$ is arbitrarily small since $r$ is controlled by $\chi_r$'s support, and we can make $\alpha$ as large as we'd like.

Letting $v=G_r T_1$ and commuting $D_r$ using lemma \ref{commutator_lemma2} and absorbing these weightless lower order terms into $\tilde{R}$ yields:
\bea
\langle B_iD_r u, u\rangle &=& \langle R_i D_r v, v\rangle+ \langle \tilde{R}u,u\rangle\\
&\leq& \|R_i D_r v\| \|v\|
\eea
by cauchy-schwarz.

$T_{1}$ is elliptic so it has a parametrix $T_{-1}$ and $Id=T_1T_{-1}+F$. Using this and letting $v=G_r T_1 u$:
\bea
\lefteqn{\|R_i D_r v\|} \\
&=& \|R_i(T_1T_{-1}+F)D_rv\|\\
&\leq& \|(R_iT_1)(T_{-1}D_rv\| + \|R_iFD_rv\|\\
&\leq& 2\sup|\sigma(R_i)\| \|T_{-1}D_rv\| + \|R_i'T_{-1}D_rv\|+ \|R_iFD_rv\|
\eea
for some $R_i'\in \Psi_b^{-1}(M)$.

A useful proposition for cross-terms $\|u\|\|v\|$ is:

\begin{prop}\label{arbitrage}
$\|u\|\|v\| \leq \frac{1}{\gamma}\|u\|^2 + \gamma \|v\|^2$
\end{prop}
\begin{proof} Expand:
$\left(\frac{1}{\sqrt{2\gamma}} \|u\| - \sqrt{2\gamma}\|v\|\right)^2 \ge 0.$
\end{proof}

Applying this to $\|R_i D_r v\| \|v\|$:

\bea\label{before_lemm}
\langle R_iD_rv, v\rangle &\leq& 2\sup|\sigma(R_i)\|T_{-1}D_r v\| \|v\| + 2\gamma \|v\|^2 \\
&&+ \frac{1}{\gamma}\|R_i'T_{-1}D_rv\|^2+\frac{1}{\gamma}\|FD_rv\|^2
\eea

We still want to absorb this term into $\|G_tD_t u\|^2$, and this will require that we replace $D_r$ in our expression by $D_t$. This requires the following crucial lemma which lets us do this without introducing non-domain bounded remainders:

\begin{lem}\cite{melrose_wunsch_vasy}\label{swap_lemma}
For $K\subset U\subset ^bS^*M$, $K$ compact and $U$ open and $A_r \in \Psi_b^{s-1}(M)$ for $r\in(0,1]$ a basic family with $\emph{WF}_b'(A_r)\subset k$ bounded in $\Psi^s_{b\infty}$. Then there exists $G\in \Psi_b^{s-1/2}(M)$ with $\emph{WF}_b'(G)\subset U$ and $C_0>0$ such that when $\square u = 0$,
\bea
\left\|\int_M (|d_M A_ru|^2-|D_tA_ru|^2) \right\| \leq C_0(\|u\|_{\tilde{\mathcal{D}}_loc}^2 + \|Gu\|_{\mathcal{D}}^2)
\eea
\end{lem}
\begin{proof}For a proof, please refer to \cite{melrose_wunsch_vasy}. Note that the $r$ parameter is required for a later regularity approximation argument.
\end{proof}

This allows us to estimate the $D_r$ term as follows:
\bea
\|D_r A_ru\|^2 &\leq& \left\|\int_M (|d_M A_ru|^2-\|D_tA_ru\|^2+\|D_tA_ru\|^2) \right\| \\
&\leq& C_0(\|u\|_{\tilde{\mathcal{D}}_{\text{loc}}}^2 + \|Gu\|_{\mathcal{D}}^2)+ \|D_tA_ru\|^2
\eea
From the expression (\ref{before_lemm}), only the first two terms aren't bounded by hypothesis on their order:
\be
2\sup|\sigma(R_i)\|T_{-1}D_r G_r T_1 u\| \|G_r T_1 u\| + 2\gamma \|G_r T_1 u\|^2
\ee
In the first term we can commute $D_r$ using our lemma \ref{commutator_lemma2} at the cost of a residual term which can be bounded by background regularity. We note the that first term is a cross term again so we may use \ref{arbitrage} as before to express this in terms of squared norms. Then after using the lemma, the expression is bounded by $2C\sup|\sigma(R_i)|\|G_r D_t u\|^2$  plus lower order terms. We are left with a $\gamma \|G_r T_1 u\|^2$ term which group with the second term.

In the second term, we can commute as desired and we note that $T_1$ has a symbol that asymptotically matches $D_t$. Thus, $2\gamma \|G_r T_1 u\|^2$ can be expressed as $2\gamma \|G_r D_t u\|^2$ with a residual term that is bounded.

The term $2C\sup|\sigma(R_i)|\|G_r D_t u\|^2$ can be made arbitrarily small by $2\sup|\sigma(R_i)|$ which becomes vanishingly small for large $\alpha$ and by $\gamma$ which can clearly be as small as we want in \ref{arbitrage}. They can therefore be absorbed into the existing $\|G_r D_t u\|^2$ term with coefficient $1$ without changing its sign.

The other term involving $B_4$ and $D_r$ is handled similarly.

This leaves the term $\langle \frac{f}{r^2}C_7 u, u\rangle$ from our potential. This may seem lower order, but it is not sufficiently low to be absorbed into our domain norm under the wavefront assumption we had. (That would require $C_7 \in \Psi_b^{-2}(M)$.) To absorb this term into existing terms, we factor $C_7$ into the same terms as we did for $C_i$ with the symbol $4ia\partial_\xi (a)$:
\be
C_7 = -RT_{-1}^*B^*BT_{-1} - RG^*G+E_1+E_2+R'
\ee
with $\sigma(R) = r \partial_r f(r,\theta)$ All terms except $RT^*B^*BT$ and $RG^*G$ end up being bounded for the assumptions stated before. $G$ decomposes into three terms as before $\sum \tilde{G}_i^*\tilde{G}_i$. After commuting (without adding weights), by cauchy-schwarz we can bound the pairings by:
\bea
\langle\frac{1}{r} RBT_{-1}u, \frac{1}{r}BT_{-1}u\rangle &\leq& \|\frac{1}{r} RBT_{-1}u\|\|\frac{1}{r} BT_{-1}u\| \\
\langle\frac{1}{r} R\tilde{G}u, \frac{1}{r}\tilde{G}u\rangle &\leq& \|\frac{1}{r} R\tilde{G} u\|\|\frac{1}{r} \tilde{G} u\|
\eea
and we would like show these terms are arbitrarily small in order to absorb them into the main terms in the pairing. We can bound these terms in the pairing in terms of squared norms using \ref{arbitrage} as before. Then, by Hardy's inequality, each of these norm expressions are bounded by
\bea
\gamma^2\left(\|D_r BT_{-1}u\|^2+\|BT_{-1}u\|^2\right)\\
\gamma^2 \left(|D_r \tilde{G}_ru\|^2+\|BT_{-1}u\|^2 \right)
\eea
We may exchange the first terms in each expression above for terms involving $D_t$ and a lower order term finite in the domain using the lemma above, so they become
\bea
\gamma^2 \left(\|D_t BT_{-1}u\|^2 + \|BT_{-1}u\|^2\right) +\emph{l.o.t.}\\
\gamma^2 \left(\|D_t \tilde{G}u\|^2 + \|BT_{-1}u\|^2\right) + \emph{l.o.t.}
\eea
respectively. Making $\gamma$ small controls these terms.

After applying \ref{swap_lemma}, ours is a question of making the following terms arbitrarily small for potentially large $\frac{1}{\gamma^2}$ values:
\bea
\frac{1}{\gamma^2}\|D_t RBT_{-1}u\|^2\\
\frac{1}{\gamma^2}\|D_t R\tilde{G} u\|^2
\eea
We would like to absorb these into corresponding terms involving $B D_t$ and $G D_t$ in our main expression whose signs we must maintain in order to preserve signs. The rest of the terms are bounded by being weightless and lower in order.

These will indeed be small since the supremum of $\sigma(R)$ can be made arbitrarily small by our control in $r$. With out loss of generality, we consider the term involving $\tilde{G}$ after (weightless) commuting as usual. 
\begin{lem}
For $R\in\Psi_b^0$, we have $M_\delta\in\Psi_b^0(M)$ $R',R''\in\Psi_b^{-1}$ the inequality:
\be
\|R D_t \tilde{G}v\| \leq 2\sup_{supp(\sigma(B))} |\sigma(R)|\|D_t Gv\| + \|R''M_\delta D_t \tilde{G}v\| + \|RR'D_tv\|
\ee
\end{lem}
\begin{proof}
To see this, create $M_\delta = Op(\chi_r\cdots)\in \Psi_b^0(M)$ using smooth cutoffs that equal 1 on the support of $\sigma(\tilde{G})$ so that $\sigma(M)\sigma(\tilde{G}) = \sigma(\tilde{G})$. (We could for instance double the width of the domain profile for each $\chi$.) and in particular we are concerned with cutting off in $r$ to bound the size of $r\partial_r(a)$. Then $M\tilde{G} = \tilde{G} - R'$ for some $R' = \Psi_b^{-3/2}(M)$. Thus we can express the above as:
\bea
\lefteqn{\|R D_t \tilde{G} v\|} \\
&\leq& \|R D_t M Gv\| + \|RD_tR'v\|\\
&\leq& \sup_{\emph{supp} (\sigma(\tilde{G}))}|\sigma(R)\| \|D_t\tilde{G}v\| + \|R''D_tM\tilde{G}v\| + \|RD_tR'v\|
\eea
for some $R''\in \Psi_b^{-1}(M)$.
\end{proof}

The coefficient $\sup_{\emph{supp} (\sigma(\tilde{G}))}\sigma(R) = \frac{1}{\lambda^2} (\sup_{r\in \emph{supp}(\sigma(M))} \delta \partial_r f(r,\theta))$ is clearly arbitrarily small by making $\delta$ small as $r$ is localized by the principal symbol of $G$. The dangling terms $\|RGu\|$ $\|RD_tR'u\|$ and $\|R''D_tMBT_{-1}u\|$ are of lower order and hence bounded.

Squaring $\|R D_t B T_{-1} v\|$ and $\|R D_t \tilde{G}v\|$ as presented in our lemma yields our main term and ``good sign term'' with arbitrarily small coefficients as well as cross terms with of the aforementioned with bounded terms. The cross terms can again be bounded by \ref{arbitrage}. Since the resulting main and ``good sign'' term can be made arbitrarily small, we can absorb them into the existing main and ``good sign'' terms in the main argument while maintaining their signs. (All absorption steps are clearly compatible with all other absorptions since we could pick our symbols arbitrarily small for each.)

Returning to our main pairing:
\bea
-i\langle\lbrack A^*A, \square + \frac{f(r,\theta)}{r^2}\rbrack u, u \rangle &=& \langle -W (\Box+\frac{f(r,\theta)}{r^2})u, u\rangle\\
&&-\langle D_tTBu, D_tTBu\rangle \\
&&- C\sum \langle \tilde{G}_i^* D_t u, \tilde{G}_i D_tu \rangle \\
&&+ \langle E_1 u, u \rangle + \langle E_2 u, u\rangle + \langle R u, u\rangle
\eea
where $C\sum \langle G_i^* D_t u, G_i D_tu \rangle$ is positive from our argument just now. The expression: $\langle D_tTBu, D_tTBu\rangle$ is thus bounded by the absolute values of the remaining terms which are finite. A final application of the lemma above allows us to bound $\|d_{\mathbb{R}^n} TBu\|$ by $\|D_t TBu\|$, and therefore the entire domain norm with $1/2$ higher order. However, the pairings may not always make sense for insufficiently regular $u$.

We finally address the required regularity of $u$ for these pairings to make sense. We resolve this by using an approximating argument using sequence of smoothing operators $\Lambda_r$ quantized from the symbol:
\be
|\tau|^s (1+\alpha|\tau|^2)^{-s/2}
\ee

Letting $A_r = A \Lambda_r$:
\bea
\lefteqn{-i\langle\lbrack A_r^*A_r, \square +\frac{f(r,\theta)}{r^2}\rbrack u, u \rangle=}\\
&& \langle -W (\square+\frac{f(r,\theta)}{r^2})\Lambda_ru, \Lambda_ru\rangle\\
&&-\langle D_tTB\Lambda_ru, D_tTB\Lambda_ru\rangle \\
&&- \sum \langle \tilde{G}^*_i D_t \Lambda_ru, \tilde{G}D_t\Lambda_ru \rangle \\
&&+ \langle E_1 \Lambda_ru, \Lambda_ru \rangle + \langle E_2 \Lambda_ru, \Lambda_ru\rangle + \langle R \Lambda_ru, \Lambda_ru\rangle
\eea

The left hand side is:
\bea
\lefteqn{i\langle [A_r^*A_r,\square+\frac{f(r,\theta)}{r^2}]u,u\rangle}\\
&=& \langle A_r^*A_r\left(\square +\frac{f(r,\theta)}{r^2}\right)u, u\rangle - \langle\left(\square +\frac{f(r,\theta)}{r^2}\right) A_r^*A_r u,u\rangle \\
&=& \langle A_r\left(\square +\frac{f(r,\theta)}{r^2}\right) u, A_r u\rangle - \langle A_r u, A_r \left(\square +\frac{f(r,\theta)}{r^2}\right) u\rangle
\eea
which is 0 as $u$ is a solution.

The ``good sign'' term has the same sign as our main term so it does not hurt our inequality and we can ignore it. The other terms are bounded because of WF hypothesis in the case of $E_1$, elliptic regularity in the case of $E_2$, and inductive hypothesis for $R$. Therefore, the $L^2$ norm $\| D_tTB\Lambda_ru \|$ is uniformly bounded as $r\rightarrow 0$. This means that weak compactness holds and that some subsequence of $D_tT B\Lambda_r u$ converge to weakly in $L^2$, or strongly to some distribution. However, the Riesz representation theorem guarantees that this is some $L^2$ function, and by uniqueness of distributions, we have that $D_tTB\Lambda_0u$ is an $L^2$ function. 

Hence $\|D_t T B u\|^2$ is bounded and we have no wavefront set of order $\frac{1}{2}+s$. We can inductively repeat this argument by shrinking the elliptic set of B arbitrarily much at each step and conclude $q_0 \notin \emph{WF}_{b,\tilde{\mathcal{D}}}^{\infty}(u)$.
$\Box$

\end{document}